\documentclass{amsproc}

\usepackage{amsmath}
\usepackage{amsfonts}
\usepackage[dvips]{graphicx}
\usepackage{amssymb}
\usepackage{amsbsy}
\usepackage{amsthm}
\usepackage{graphics}
\usepackage{color}
\usepackage{textcomp}
\usepackage{ mathrsfs }
\usepackage{tikz}
\usetikzlibrary{arrows.meta}

\makeatletter
\@namedef{subjclassname@2010}{%
\textup{2010} Mathematics Subject Classification}
\makeatother

\numberwithin{equation}{section}

\newtheorem{theorem}{Theorem}[section]

\newtheorem{lemma}[theorem]{Lemma}

\newcommand{\ttt}{\mathscr{T}}
\newcommand{\ann}{\mathbb{A}}
\newcommand{\anns}{\ann{(\rmin,\rplus)}}
\newcommand{\annr}{\ann(r^-,r^+)}

\newcommand{\scd}{T^{\prime*}}
\newcommand{\cd}{T^{\prime}}

\newcommand{\hil}{\mathcal{H}}
\newcommand{\chil}{\mathscr{H}}

\newcommand{\ddd}{\mathcal{D}}
\newcommand{\natu}{\mathbb{N}}

\newcommand{\bou}{\boldsymbol B(\hil)}
\newcommand{\sbou}{\boldsymbol{B}}
\newcommand{\boue}{\boldsymbol{B}(E)}

\newcommand{\spec}{\sigma(T)}

\newcommand{\rad}{r(T)}

\newcommand{\nul}{\mathcal{N}(}

\newcommand{\comp}{\mathbb{C}}
\newcommand{\bil}{S_{\mathbf{\lambda}}}
\newcommand{\disc}{\mathbb{D}}
\newcommand{\elu}{L^2(\mu)}
\newcommand{\pe}{P_E}
\newcommand{\mul}{\mathscr{M}_z}
\newcommand{\com}{C_{\phi,w}}
\newcommand{\cdcom}{C_{\phi^\prime,w}}
\newcommand{\szift}{S_\lambda}
\newcommand{\szifto}{S_{\lambda\to(\omega)}}

\newcommand{\cdszift}{S_{\lambda^\prime}}
\newcommand{\jad}{\kappa_\chil}

\theoremstyle{definition}
\newtheorem{ex}[theorem]{Example}

\newcommand*{\Le}{\leqslant}
\newcommand*{\Ge}{\geqslant}
\newcommand{\ran}{\mathcal{R}(}
\newcommand{\la}{\langle}
\newcommand{\ra}{\rangle}
\newcommand{\gwon}{[E]_{T^*,\cd}}
\newcommand{\sgwon}{[E]_{{\szift^*},{\cdszift}}}
\newcommand{\cgwon}{[E]_{{\com^*},{\cdcom}}}

\DeclareMathOperator{\gen}{Gen}
\DeclareMathOperator{\lin}{lin}
\DeclareMathOperator{\card}{card}
\DeclareMathOperator{\czil}{Chi}

\DeclareMathOperator{\des}{Des}
\DeclareMathOperator{\ro}{root}
\DeclareMathOperator{\parr}{par}
\DeclareMathOperator{\intt}{int}

\newcommand{\cycle}{\mathscr{C}_{\varphi}}

\newcommand{\rplus}{r^+_{w,\phi}}
\newcommand{\rmin}{r^-_{w,\phi}}
\newcommand{\kk}{k_{\phi(E)}}

\newcommand{\set}[1]{\left\{#1\right\}}

\newcommand{\norm}[1]{\left\Vert#1\right\Vert}

\begin{document}
\title[A Shimorin-type analytic model for left-invertible operators]{A Shimorin-type analytic model on an annulus\\ for left-invertible operators and applications}
   \author[P. Pietrzycki]{Pawe{\l} Pietrzycki}
   
   \subjclass[2010]{Primary 46E22, 47B32, 47B33; Secondary 47A10, 47B37, 47B38} \keywords{analytic model, weighted composition operator, weighted shift, weighted shift on directed three,  multiplication operator,  reproducing kernel Hilbert space, Hilbert space of holomorphic functions.}
   \address{Wydzia{\l} Matematyki i Informatyki, Uniwersytet
Jagiello\'{n}ski, ul. {\L}ojasiewicza 6, PL-30348
Krak\'{o}w}
   \email{pawel.pietrzycki@im.uj.edu.pl}
   \begin{abstract} A new  analytic model for left-invertible operators,  which extends both Shimorin's analytic model for left-invertible and analytic operators and Gellar's model for bilateral weighted shift is introduced and investigated. We show that a left-invertible operator $T$, which satisfies certain conditions can be modelled as a multiplication operator $\mul$ on a reproducing
kernel Hilbert space of vector-valued analytic functions on an annulus or a disc.  A similar result for composition
operators in $\ell^2$-spaces is established.
\end{abstract}
   \maketitle

   \section{Introduction}
   The classical Wold decomposition theorem (see \cite{wold}) states that if $U$ is isometry on Hilbert space $\hil$, then $\hil$ is the direct sum of two
subspaces reducing $U$, $\hil=\hil_u\oplus\hil_p$ such that $U|_{\hil_u}\in\sbou(\hil_u)$ is unitary and $U|_{\hil_p}\in\sbou(\hil_p)$ is unitarily equivalent to a unilateral shift.
This decomposition is unique and the canonical subspaces are defined by
\begin{align*}
  \hil_u&=\bigcap_{n=1}^\infty U^n\hil\quad \textrm{and}\quad
  \hil_p=\bigoplus_{n=1}^\infty U^nE,
\end{align*}
where $E=\nul U^*)=\hil\ominus U\hil$. The Wold decomposition theorem and results analogous to this theorem plays an important role in many areas of operator theory, including the invariant subspace problem for Hilbert spaces of holomorphic functions.
The interested reader is
referred to \cite{sloc,paga,burd,burd2,shi,kall,kall2,mand,jay}.

One of the key ideas in operator theory is that of viewing an operator  as multiplication by $z$ on a Hilbert space consisting of (vector-valued) holomorphic functions. The point is that this multiplication  operator is much easier to analyse than is the case in the original setting because of the richer structure of a space of holomorphic functions. This is its great advantage and one of the reasons why it attract attention of researchers.
An excellent example
of an interplay between weighted shift operators and analytic functions is the problem of describing all invariant subspaces of weighted shifts and the celebrated Beurling-Lax theorem.
There are numerous results in the
literature relating  analytic models  for the operators in certain classes.
We mention some selected models:
\begin{itemize}
\item[$\bullet$] the Sz.-Nagy-Foias model for contraction (see \cite{contr}),
\item[$\bullet$] the analytic model for weighted shifts (see \cite{gel}),
\item[$\bullet$] the  analytic
model  of  a  pure  hyponormal  operator $T$
with  rank  one  self-commutators
$[T^*,T]$ (see \cite{pin}),
    \item[$\bullet$] the model for the class
$\mathscr{F}$
of pure operators $T$ on a Hilbert space $\hil$ satisfying
\begin{equation*}
     \la T^mg, T^nh\ra=0,\qquad g,h\in [T^*,T]\hil,\:\: m\neq n ,\:\: m,n\in\natu,
\end{equation*}
where $[T^*,T]:=T^*T-TT^*$ (see \cite{xia4}),
\item[$\bullet$] Shimorin's analytic model for left-invertible analytic operators (see \cite{shi}),
\item[$\bullet$] the analytic model of doubly commuting contractions (see \cite{sark}).
\end{itemize}
The interested reader is referred to \cite{chav,jewe,xia,xia2,xia3,xia4} for further information.

In \cite{shi} S. Shimorin  obtain a weak analog
of the  Wold decomposition theorem, representing operator close to isometry in some sense as a direct sum
of a unitary operator and a shift operator acting in some reproducing kernel Hilbert space of vector-valued
holomorphic functions defined on a disc. The construction of
 the Shimorin's model for a left-invertible analytic operator $T\in \bou$ is as follows.
Let $E:=\nul T^*)$ and define a 
vector-valued  holomorphic  functions $U_x$ as
  \begin{equation*}
   U_x(z) =
\sum_{n=0}^\infty
(P_E{T^{\prime*n}}x)z^n,\quad z\in \disc({r(\cd)}^{-1}),
 \end{equation*}
where $\cd$
is the Cauchy dual of $T$. Then
we equip the obtained space of analytic functions $\chil:=\set{U_x:x\in \hil}$
with the inner product induced by $\hil$. The operator $ U:\hil\ni x\to U_x\in \chil$
becomes a unitary isomorphism. It turns out that the operator $T$ is unitarily equivalent to the operator $\mul$ of multiplication
by $z$ on $\chil$ and $\scd$ is unitarily equivalent to the operator $\mathscr{L}$  given by 
\begin{equation*}
     (\mathscr{L}f)(z)=\frac{f(z)-f(0)}{z}, \quad f\in\chil.
\end{equation*}
Moreover, Shimorin proved that $\chil$ is a  reproducing kernel Hilbert space in the following sense: \textit{the reproducing kernel} for $\chil$ (see \cite{shi}) is an $\boue$-valued function of two variables $\jad:\Omega\times\Omega\rightarrow \boue$ such that
   \begin{itemize}
       \item[(i)] for any $e\in E$ and $\lambda \in \Omega$
       \begin{equation*}
           \jad(\cdot,\lambda)e\in \chil,
       \end{equation*}
       \item[(ii)]for any $e\in E$, $f\in \chil$ and $\lambda \in \Omega$
       \begin{equation*}
           \la f(\lambda),e\ra_E=\la f,\jad(\cdot,\lambda)e\ra_\chil.
       \end{equation*}
   \end{itemize}
The interested reader is referred to \cite{shaf1} and \cite{shaf2} for further 
facts 
concerning 
the reproducing 
kernel  Hilbert 
space 
and 
its  multiplication 
operators. 

The substitution operation is basic to mathematics therefore
composition operators naturally appear in many  areas of mathematics. They play an important role in ergodic theory and functional analysis. The  class of composition operators is related to 
other areas of operator theory in somewhat surprising ways. S. Banach and M. Stone proved that a surjective linear isometry  $T:C(X)\to C(Y)$ is a weighted composition operator. The analogical result for the Hardy spaces $H^p(\disc)$ (with $p\Ge1$ and $p\neq 2$) was shown by Forelli in \cite{forel}. Furthermore, commutants of many analytic Toeplitz  operators are generated by composition and multiplication operators. 
The literature on this subject is vast and still growing 
(see e.g., \cite{chaf,sin,sin2,nord,nord2,whit,har,lam1,lam2,dib,emb,bur,b-j-j-s,carl,ja,dym2}).

   The class of  weighted shifts on a directed tree was introduced in \cite{memo} and intensively studied
since then (see e.g.,  \cite{9,geh,chav,chav2,planeta,bdp,adv,mart,dym2}). Z.J. Jab{\l}o\'{n}ski, I. B. Jung and
J. Stochel  realized  the  importance  of  this  class as  a  vehicle  to  collect  a  number of interesting examples  and counterexamples (see e.g., \cite{dym,9,ja,ja2,trep,triv,dodana,sque,jabl}).

  The analytic aspects of the theory of composition operators, weighted shifts  and weighted shifts on a directed tree were studied by many authors.
   As was mentioned by A.L. Shields in the paper \cite{shild} the fact that weighted shift can be  viewed as  multiplication by $z$ on a Hilbert space of formal power series
has been long folklore and this point of view was taken by R. Gellar (see \cite{gel,gel2}). He showed that
the commutant of any weighted shift operator consists of certain formal power
series in the operator, and hence that the commutant is abelian. According to \cite{shild} the  spectrum of weighted shift operator
is either 
an annulus 
or a disk. Some results on the spectrum and commutants of composition operators were obtained in \cite{rid} and \cite{carl}. 
In \cite{chav} S. Chavan and S. Trivedi showed that a weighted shift $\szift$ on  a rooted directed tree with finite branching index is analytic therefore  can be modelled as a multiplication operator $\mul$ on a reproducing
kernel Hilbert space $\chil$ of $E$-valued holomorphic functions on  a disc centered
at the origin, where $E := \nul\szift^*)$. Moreover, they proved that the reproducing kernel associated with $\chil$
is multi-diagonal.
It is worth pointing out that  the commutant and reflexivity for
$n$-tuples of multiplication operators by independent
variables
$z_1,\dots,z_n$
on a reproducing Hilbert space of
vector-valued holomorphic functions were studied in 
the paper \cite{poder}.

Recently, the analytic structure
of weighted shifts on directed trees was also studied by P. Budzy{\'n}ski,  P. Dymek, A. P{\l}aneta and M. Ptak.
In \cite{bdp} they
showed that a weighted shift on a rooted directed tree is related to a
multiplier algebra
of coefficients of analytic functions. They  used this relation to provide a
kind of functional calculus for functions from multiplier algebras and to
study spectral properties of weighted shift on a rooted directed tree. Moreover in \cite{dym} they  extended the notion of multipliers to left-invertible and analytic operators and  characterize the commutant of such operators in  terms of generalized multipliers. This line of investigation
was continued in \cite{ja4}.

In this paper, we provide a new  analytic model for left-invertible operators, which extends both Shimorin's analytic model for left-invertible  and analytic operators (see Theorem \ref{pokshi}) and Gellar's model for a bilateral weighted shift (see Example \ref{pokgel}). We show that a left-invertible operator $T$, which satisfies certain conditions can be modelled as a multiplication operator $\mul$ on a reproducing
kernel Hilbert space of vector-valued analytic functions on an annulus or a disc (see Theorem \ref{przep}).  As an application of this model, we obtain significantly improved model for weighted composition operator upon provided the symbol of this operator
has finite branching index (see Theorem \ref{modcom}). In particular,  we describe the inner and outer radius of convergence for weighted composition operators only  in terms of its
weight and symbol.



   \section{Preliminaries}\label{prel}
   In this paper, we use the following notation. The
fields of rational, real and complex
numbers are denoted by $\mathbb{Q}$,
$\mathbb{R}$ and $\mathbb{C}$, respectively. The
symbols $\mathbb{Z}$, $\mathbb{Z}_{+}$ and $\mathbb{N}$
stand for the sets of integers,
positive integers and nonnegative integers, respectively. Set 
$\disc(r)=\set{z\in \comp\colon |z|< r}$ and
$\ann(r^-,r^+)=\set{z\in \comp\colon r^-< |z|< r^+}$ for $r,r^-,r^+\in[0,\infty)$.
The expression "a countable set" means a
finite set or a countably infinite set.

All Hilbert spaces considered in this paper are assumed to be complex. Let $T$ be a linear operator in a complex Hilbert
space $\mathcal{H}$. Denote by  $T^*$  the adjoint
of $T$.  We write $\bou$ 
 for the $C^*$-algebra of all bounded operators.
The spectrum, point spectrum and spectral radius of $T\in\bou$ is denoted by $\spec$, $\sigma_p(T)$ and $\rad$
respectively.
Let $W$ be a subset of $\hil$. Then
$\lin W$, $\bigvee W$ stands for the smallest linear subspace, closed subspace generated by $W$, respectively. We use the notation
 $\la x\ra$ in place of $\lin\{x\}$, for $x\in \hil$.
Let $T \in\bou$.
We say that $T$ is \text{left-invertible} if there exists $S \in \bou$ such
that $ST = I$. The \textit{Cauchy dual operator} $\cd$ of a left-invertible operator $T\in \bou$ is defined by
\begin{equation*}
 \cd:=T(T^*T)^{-1}. 
\end{equation*}
Note that $T$
is left-invertible if and only if there exists a constant $c>0$ such that $T^*T\Ge cI$. The notion of the Cauchy dual operator has been introduced and studied by Shimorin
in the context of the wandering subspace problem for Bergman-type operators \cite{shi}. We call $T$ \textit{analytic} if $\hil_\infty:=\bigcap_{i=1}^\infty T^i\hil=\set{0}$. Let $\Omega\subset\comp$ be such that $\intt\Omega=\Omega\neq\emptyset$. A function $f:\Omega\to\hil$ is said to be \textit{holomorphic} on $\Omega$ if $f$ is differentiable. 

Let $X$ be a set and $\varphi:X\to X$. If $n\in \mathbb{Z}$, then the $n$-th iterate of $\varphi$ is given by $\varphi^{(n)}=\varphi\circ\varphi\circ\dots\circ\varphi$, $\varphi$ composed with itself $n$-times and $\varphi^{(0)}$ is identity function. For $x\in X$ the set 
   \begin{equation*}
 [x]_\varphi=\{y\in X: \text{there exist } i,j\in \natu \text{  such that  } \varphi^{(i)}(x)=\varphi^{(j)}(y) \}
   \end{equation*} 
   is
called the \textit{orbit of}  $\varphi$ containing $x$. If $x\in X$ and $\varphi^{(i)}(x)=x$ for some $i\in \mathbb{Z}_+$, then the \textit{cycle of} $\varphi$ containing $x$ is the set
\begin{equation*}
  \cycle=\{\varphi^{(i)}(x)\colon i\in \natu \}.
\end{equation*}
   Define the function $[\varphi]:X\rightarrow \mathbb{Z}$ by
   \begin{itemize}
       \item[(i)] $[\varphi](x)=0$ if  $x$ is in the cycle of $\varphi$
       \item[(ii)] $[\varphi](x^*)=0$, where $x^*$ is a fixed element of orbit $F$ of $\varphi$ not containing a cycle,
       \item[(iii)]
     $[\varphi](\varphi(x))=[\varphi](x)-1$ if $x$ is not in a cycle of $\varphi$.
   \end{itemize}We set
    \begin{equation*}
    \gen_\varphi{(m,n)}:=\{x\in X\colon m\Le[\varphi](x)\Le n\}
\end{equation*}for $m,n\in\mathbb{Z}$.
Let $(X,\mathscr{A},\mu)$ be a $\mu$-finite measure space, $\varphi: X \rightarrow X$ and $w:X \rightarrow \comp$ be
measurable transformations.
By a \textit{weighted composition 
operator} $\com$ in $\elu$ we mean a mapping
\begin{align*}
\ddd(\com)&=\{f\in \elu : w(f\circ\varphi)\in \elu\},
 \\\notag
\com f&=w(f\circ\varphi),\quad f \in\ddd(\com).
\end{align*} 
We call $\varphi$ and $w$ the \textit{symbol} and the \textit{weight} of $\com$ respectively.

Let us recall some useful properties of composition operator we need
in this paper:

\begin{lemma}\label{podst}Let $X$ be a countable set, $\varphi: X \rightarrow X$ and $w:X \rightarrow \comp$ be
measurable transformations. If $\com\in \sbou(\ell^2(X))$, then for any $x\in X$ and $n\in \mathbb{Z}_+$
\begin{itemize}
\item[(i)]$\com^*e_x=\overline{w(x)}e_{\varphi(x)}$
\item[(ii)] $\com e_x=\sum_{y\in\varphi^{-1}(x)}w(y)e_y$,
 \item[(iii)]$\com^{*n} e_x=\overline{w(x)w(\varphi(x))\cdots w(\varphi^{(n-1)}(x))}e_{\varphi^{(n)}(x)}$,
       \item[(iv)]$\com^n e_x=\sum_{y\in\varphi^{-n}(x)}w(y)w(\varphi(y))\cdots w(\varphi^{(n-1)}(y))e_y$,
     \item[(v)] $\com^*\com e_x=\Big(\sum_{y\in\varphi^{-1}(x)}|w(y)|^2\Big)e_x$.
   \end{itemize}
  
   \end{lemma}
   \begin{proof}
   
(i) and (ii) See \cite[page 633]{carl}. 

(iii) and (iv) Apply (i), (ii) and induction on
$n$.

(v) This follows from (i) and (ii).
   \end{proof}
    We now describe Cauchy dual of weighted composition operator.
    
   \begin{lemma}\label{cdcom}Let $X$ be a countable set, $\varphi: X \rightarrow X$ and $w:X \rightarrow \comp$ be
measurable transformations. If $\com\in \sbou(\ell^2(X))$ is left-invertible operator, then the Cauchy dual $\com^\prime$ of $\com$ is also a weighted composition operator $\cdcom$ with the same symbol $\varphi:X\to X$ and weight $w^\prime:X\to\comp$ defined by
\begin{equation*}
  w^\prime(x):=  \frac{w(x)}{\Big(\sum_{y\in\varphi^{-1}(\varphi(x))}|w(y)|^2\Big)}.
\end{equation*}
   \end{lemma}
   \begin{proof}
 This is a direct
consequence of assertions (i) and (ii) of Lemma \ref{podst}.
   \end{proof}
 Let $\ttt = (V; E)$   be a directed tree ($V$ and
$E$ are the sets of vertices and edges of $\ttt$,
respectively). For any vertex $u \in V$ we
put $\textup{Chi}(u) = \{v \in V : (u, v) \in E\}$. Denote by $\parr$ the partial  function from $V$ to $V$ which assigns to a vertex $u$  a unique $v\in  V$ such that $(v,u)\in E$.  A vertex $u \in V$ is called a root of $\ttt$ if $u$ has no parent. If $\ttt$ has a root, we denote it by
$\textrm{ root}$. Put $V^{\circ}= V
\setminus\{\textrm{root}\}$ if $\ttt$ has a root and
$V^{\circ}=V$ otherwise.   The Hilbert
space of square summable complex functions on $V$
equipped with the standard inner product is denoted by
$\ell^2(V )$. For $u \in V$, we define $e_u \in
\ell^2(V)$ to be the characteristic function of the
 set $\{u\}$. It turns out that the set $\set{e_v}_{v\in V}$ is an orthonormal basis of $\ell^2(V)$. We put $V_\prec:=\set{v\in V: \card(\czil(V))\Ge2}$ and call the a member of this set a \textit{branching vertex} of $\ttt$

Given a system $ \lambda=\{\lambda_v \}_{v \in
V^{\circ}} $ of complex numbers, we define the
operator $S_\lambda$ in $\ell^2(V)$, which is called a
\textit{weighted shift} on $\ttt$ with weights
$\lambda$, as follows
   \begin{equation*}
\ddd(S_{\lambda}) =\{ f\in \ell^2(V ):
\varLambda_{\ttt} f \in \ell^2(V )\}\quad \textup{and}
\quad S_{\lambda} f = \varLambda_{\ttt} f \quad
\textup{for} \quad f \in \ddd(S_\lambda),
   \end{equation*}
where
   \begin{displaymath}
(\varLambda_{\ttt} f)(v) =\left\{\begin{array}{ll}
\lambda_v f(\textup{par}(v)) & \textrm{if  $v \in V^{\circ},$}\\
0 & \textrm{otherwise}.\end{array} \right.
 \end{displaymath}

\begin{lemma}[(Proposition 3.5.1 \cite{memo})] \label{ker}  If $\szift$ is a densely defined weighted shift on a directed tree $\ttt$
with weights  $ \lambda=\{\lambda_v \}_{v \in
V^{\circ}} $, then
\begin{equation*} 
\nul\szift^*) = \left\{ \begin{array}{ll}
\la e_{\ro} \ra\oplus\bigoplus_{u\in V_\prec}(\ell^2(\czil(u))\ominus\la \lambda^u\ra) & \textrm{if $\ttt$  has a root,}\\
\bigoplus_{u\in V_\prec}(\ell^2(\czil(u))\ominus\la \lambda^u\ra) & \textrm{
otherwise,}
\end{array} \right.
\end{equation*}
where $\lambda^u\in\ell^2(\czil(u))$ is given by $\lambda^u:\ell^2(\czil(u))\ni v\to\lambda_v\in \comp$.
\end{lemma}
A subgraph of
a directed tree $\ttt$ which itself is a directed tree will be called a subtree of $\ttt$. We refer the reader to \cite{memo} for more details on
weighted shifts on directed trees.


   \section{Analytic model}\label{mymodel}
   
   This section provides an analytic  model for  left-invertible operators. We show that a
left-invertible operator, which satisfies certain conditions can be modelled as a multiplication
operator
on  a  reproducing  kernel  Hilbert  space  of  vector-valued  analytic  functions  on  an
annulus or a disc.

   Let $T\in \bou$  be a left-invertible operator and $E$ be a subspace of $\hil$ denote by $\gwon$ the  following subspace of $\hil$:
\begin{equation*}
\gwon:=\bigvee\big(\{T^{*n}x\colon x\in E, n\in \natu\}\cup\{{\cd}^n x\colon x\in E, n\in \natu\}\big),
\end{equation*}
where $\cd$ is the Cauchy dual of $T$. 

To avoid the repetition, we state the following assumption which will be used
frequently in this section.
   \begin{align} \tag{$\clubsuit$} \label{li}
\begin{minipage}{70ex} 
The operator  $T\in \bou$ is left-invertible   and $E$ is a closed subspace of $\hil$ such that $\gwon=\hil$.
\end{minipage}
    \end{align}

Suppose \eqref{li} holds. In this case, we may construct a Hilbert $\chil$
 associated with $T$, of formal Laurent series with vector  coefficients. We proceed as follows. For
each $x \in\hil$, define a formal Laurent series $U_x$ with vector  coefficients  as  
   \begin{equation}\label{mod}
       U_x(z) =\sum_{n=1}^\infty
(P_ET^{*n}x)\frac{1}{z^n}+
\sum_{n=0}^\infty
(P_E{T^{\prime*n}}x)z^n.
 \end{equation}
   
 Let $\chil$ denote the vector space of formal Laurent series with vector  coefficients of the form $U_x$, $x \in \hil$.  
   Consider the map $U:\hil\rightarrow \chil$ defined by $Ux=U_x$. 
As shown in Lemma \ref{kernel} below, by the assumption $U$ is injective. In particular, we may equip the space
$\chil$ with the inner product induced from $\hil$, so that $U$ is unitary.
Observe that every $f\in\chil$ can be represented as follows
 \begin{equation*}
f(z)=\sum_{n=-\infty}^\infty
\hat{f}(n)z^n,
 \end{equation*}
where
\begin{equation*}
 \hat{f}(n) = \left\{ \begin{array}{ll}
\pe{\scd}^nU^*f & \textrm{if $n\in\natu$,}\\
\pe T^{-n}U^*f & \textrm{if $n\in\mathbb{Z}\setminus\natu$.}
\end{array} \right.
 \end{equation*}
\begin{lemma}\label{kernel} Suppose  \eqref{li} holds and $\chil$, $U$ are as above. Then 
$\nul{U)}=\{0\}$. 
\end{lemma}
\begin{proof}
Suppouse that $x\in \hil$ is such that 
\begin{equation*}
 \pe T^nx=0 \quad\text{and} \quad \pe {\scd}^nx=0,\quad n\in \natu.
\end{equation*}
Then for every $y\in E$
\begin{equation*}
 \la T^nx,y\ra=0 \quad\text{and} \quad \la {\scd}^nx,y\ra=0,\quad n\in \natu.
\end{equation*}
  This implies
  \begin{equation*}
 \la x,T^{*n}y\ra=0 \quad\text{and} \quad \la x,{\cd}^ny\ra=0,\quad n\in \natu.
\end{equation*}
We see that the above condition is equivalent to the following one
\begin{equation*}
    x \:\bot \:\gwon(=\hil).
\end{equation*}
This completes the proof.

\end{proof}
As shown below, the operator $T$ is unitary equivalent to the operator $\mul:\chil\to\chil$ of multiplication
by $z$ on $\chil$ given by 
\begin{equation*}
    (\mul f)(z)=zf(z),\quad f\in\chil,
\end{equation*}
and operator $\scd$ is unitary equivalent to the operator $\mathscr{L}:\chil\to\chil$ given by
\begin{equation*}
   (\mathscr{L}f)(z)=\frac{f(z)-(P_{\nul \mul^*)}f)(z)}{z}, \quad f\in\chil.
\end{equation*}
\begin{theorem}\label{przep}
Suppose \eqref{li} holds. Then the following assertions are valid:
\begin{itemize}
    \item[(i)] $UT=\mul U$,
    \item[(ii)] $U\scd =\mathscr{L}U$.
\end{itemize}
\end{theorem}
\begin{proof}
(i) Let $x\in \hil$. Applying \eqref{mod} to operator $T$ and vector $Tx$, we see that
\begin{align*}
(UTx)(z)&=\sum_{n=1}^\infty
(P_ET^{n}Tx)\frac{1}{z^n}+
\sum_{n=0}^\infty
(P_E{T^{\prime*n}}Tx)z^n
\\&=\sum_{n=1}^\infty
(P_ET^{n+1}x)\frac{1}{z^n}+(P_ETx)+
\sum_{n=1}^\infty
(P_E{T^{\prime*n-1}}x)z^n
\\&=z(Ux)(z).
\end{align*}

(ii) Since
\begin{align*}
(U\scd x)(z)&=\sum_{n=1}^\infty
(P_ET^{n}\scd x)\frac{1}{z^n}+
\sum_{n=0}^\infty
(P_E{T^{\prime*n+1}}x)z^n\\&=\sum_{n=1}^\infty
(P_ET^{n-1}(I-P_{\nul T^*)}) x)\frac{1}{z^n}+
\sum_{n=0}^\infty
(P_E{T^{\prime*n+1}}x)z^n\\&=\frac{(Ux)(z)-(UP_{\nul T^*)}x)(z)}{z}
\end{align*}
the proof is complete.
\end{proof}

Now we  show that in the case of a left-invertible and analytic operators
  our analytic model with $E:=\nul{T^*)}$  coincides with the Shimorin's analytic model.
\begin{theorem}\label{pokshi}
Let $T\in\bou$ be left-invertible and analytic, $\chil_1$, $U_1$ be the Hilbert space and the unitary map  construceted in \eqref{mod} with $E:=\nul T^*)$ and $\chil_2$, $U_2$ be the Hilbert space and the unitary map obtained in Shimorin's construction. Then $\chil_1=\chil_2$ and $U_1=U_2$.
\end{theorem}
\begin{proof}
Set $\hil_\infty:=\bigcap_{i=0}^\infty T^i\hil$. By \cite[Proposition 2.7]{shi},  $\hil_\infty^\perp=[E]_{\cd}$. Since $T$ is analytic $\hil_\infty=\{0\}$, we see that $[E]_{\cd}=\hil$. Therefore, condition \eqref{li} is satisfied.  By kernel-range decomposition, $P_{\nul T^*)}T^n=0$ for $n\in \mathbb{Z}_+$. Hence, the first sum in \eqref{mod} vanishes. This completes the proof.
\end{proof}
Now we describe how to obtain a collection of subspaces of $\hil$ with property \eqref{li} from a single subspace with this property.
\begin{theorem}
Suppose \eqref{li} holds.  Then for every $m\in\natu$ the following assertions hold:
\begin{itemize}
    \item[(i)]${\cd}^mE$  is a closed supspace and $[{\cd}^m E]_{T^*,\cd}=\hil$,
    \item[(ii)]the mapping $\Phi_m:\chil_0\to\chil_m$ defined by
    \begin{equation*}
     \Phi_m\Big( \sum_{n=-\infty}^{\infty}a_nz^n\Big)=\sum_{n=-\infty}^{\infty}(V_ma_{m+n})z^{n},\quad \sum_{n=-\infty}^{\infty}a_nz^n\in \chil_0
    \end{equation*}
    is a unitary isomorphism, where $\chil_k$ for $k\in \natu$ is the Hilbert space construceted in \eqref{mod} with subspace  ${\cd}^kE$ and  $V_k:E\to{\cd}^k E$ for $k\in \natu$ is defined by,
    \begin{equation*}
    Ve=P_{{\cd}^kE}T^ke, \qquad e\in E.
    \end{equation*}
\end{itemize}
\end{theorem}
\begin{proof}
 
(i) Since $T^*\cd=I$,  we get that $[{\cd}^m E]_{T^*,\cd}=\gwon$, $m \in \natu$. This in turn
implies that $[{\cd}^m E]_{T^*,\cd}=\hil$.  The operator ${\cd}^m$ is left-infertible and hence bounded below. This implies that the subspace  ${\cd}^mE$ is closed.

(ii) We will denote by $U_k$, $k\in \natu$  the  unitary operator of the form \eqref{mod} between $\hil$ and $\chil_k$. Fix $m\in \natu$.
First, we note that for every $e\in E$ and $m,n\in\natu$, we have
\begin{equation}\label{pomoc}
    \la T^nx,{\cd}^me\ra= \left\{ \begin{array}{ll}
\la T^{n-m}x,e\ra  & \textrm{if $n\Ge m$,}\\
\la{\scd}^{m-n} x,e\ra & \textrm{if $n<m$.}
\end{array} \right.
\end{equation}
Take $y\in {\cd}^{m}E$. Then there exists $e\in E$ such that $y={\cd}^{m}e$. Employing \eqref{pomoc}, we verify that 
\begin{align*}
  \la \widehat{(U_mx)}(-n),y\ra&=\la P_{{\cd}^mE}T^nx,y\ra=\la T^nx,{\cd}^me\ra
 \overset{\eqref{pomoc}}{=}\la \widehat{(U_0x)}(m-n),e\ra \\&=\la \widehat{(U_0x)}(m-n),T^{*m}{\cd}^me\ra=\la T^m\widehat{(U_0x)}(m-n),{\cd}^me\ra\\&=\la P_{{\cd}^mE}T^m\widehat{(U_0x)}(m-n),y\ra,
\end{align*}
for $n\in\natu$. This implies that 
\begin{equation}\label{wsu}
 \widehat{(U_mx)}(-n)=   P_{{\cd}^mE}T^m\widehat{(U_0x)}(m-n), \qquad n\in\natu.
\end{equation}
Arguing as above, we deduce that
\begin{align*}
  \la \widehat{(U_mx)}(n),y\ra&=\la P_{{\cd}^mE}{\scd}^nx,y\ra=\la {\scd}^nx,{\cd}^me\ra
{=}\la {\scd}^{n+m}x,e\ra
\\&=\la \pe{\scd}^{n+m}x,T^{*m}{\cd}^me\ra=\la T^{m}\pe{\scd}^{n+m}x,{\cd}^me\ra\\&=\la P_{{\cd}^mE}T^m\widehat{(U_0x)}(m+n),y\ra, 
\end{align*} 
for $n\in\natu$. As a consequence, we see that
\begin{align*}
   \widehat{(U_mx)}(n)= P_{{\cd}^mE}T^m\widehat{(U_0x)}(m+n), \qquad n\in\natu.
\end{align*} 
This and \eqref{wsu} imply that $\Phi_m$ is an isomorphism. Since the Hilbert space structure on $\chil_k$ for $k\in\natu$ is induced from $\hil$, we deduce that $\Phi_m$ is unitary. This completes the proof.
\end{proof}
For left-invertible operator $T\in\bou$, among all subspaces satisfying condition \eqref{li} we  distinguish those subspaces $E$ which satisfy the following condition
\begin{equation}\tag{$\spadesuit$}\label{prep} E\perp T^n E \qquad\text{and}\qquad E\perp {\cd}^n E, \qquad n\in \mathbb{Z}_+.
\end{equation}
A similar condition was studied in the contex of 2-isometries in \cite{ana} where analog of Wold decompositions was obtained.
\begin{theorem}Suppose \eqref{li} holds. Then the following assertions hold:

\begin{itemize}

\item[(i)] if additionaly \eqref{prep} holds, then $U(E)$ is a copy of $E$ in $\chil$, the subspace consisting of constant
functions; moreover, $E$-valued polynomials in $z$ are included in $\chil$,

\item[(ii)] $(\mul^*\mul)^{-1}\mul^*=\mathscr{L}$.
\end{itemize}
\end{theorem}
\begin{proof}
(i) This is obvious.

(ii) Fix any $x\in \hil$.
Combining 
Theorem \ref{przep} and the kernel-range decomposition, we deduce that
\begin{align*}
     (\mathscr{L}\mul Ux)(z)&=\frac{(\mul Ux)(z)-(UP_{\nul T^*)}U^{-1}\mul Ux)(z)}{z}\\&=\frac{z (Ux)(z)-(UP_{\nul T^*)}Tx)(z)}{z}=(Ux)(z),
\end{align*}
which means that $\mathscr{L}$ is a left-inverse of $\mul$. Since
\begin{equation*}\mathscr{L}\mul=I\quad \textrm{and}\quad(\mul^*\mul)^{-1}\mul^*\mul=I,
\end{equation*}   we see
that $\mathscr{L}|_{\ran \mul)}=(\mul^*\mul)^{-1}\mul^*|_{\ran \mul)}$. 
One can verify that \begin{equation*}
    \mathscr{L}|_{\nul \mul)}=(\mul^*\mul)^{-1}\mul^*|_{\nul \mul)},
\end{equation*} which completes the proof.
\end{proof}
Now we shall discuss the extent to which our formal Laurent series actually represent analytic functions. If the series \eqref{mod} is convergent in $E$ on $\Omega\subset\comp$ for every $x\in\hil$, then based upon Lemma \ref{zasid} below we regard the Hilbert space $\chil$   as a space  of  vector-valued  holomorphic  functions on $\Omega$ by identifying each formal Laurent  series \eqref{mod} with the function
\begin{equation*}
\tilde{U}_x\colon\Omega\ni z\to\sum_{n=1}^\infty
(P_ET^{*n}x)\frac{1}{z^n}+
\sum_{n=0}^\infty
(P_E{T^{\prime*n}}x)z^n \in E.
\end{equation*}
\begin{lemma}\label{zasid}
Let $\sum_{n=-\infty}^{\infty}a_nz^n$ be the formal Laurent  series which represent constant zero function on open and nonempty subset $\Omega\subset \comp$. Then $a_n=0$,
$n\in\mathbb{Z}$.
\end{lemma}
\begin{proof}
Take $e\in E$. Then the function
    $\Omega\ni z\to \sum_{n=-\infty}^{\infty}\la a_n,e\ra z^n \in \comp$ is holomorphic, on the one hand,
and identically equal to zero, on the other. By Identity theorem $\la a_n,e\ra=0$, $n\in\mathbb{Z}$. This shows that $ a_n=0$, $n\in\mathbb{Z}$ and hence completes the proof.
\end{proof}
\begin{theorem}
Suppose \eqref{li} holds. Let
\begin{align*}r^+:=&\inf_{x\in\hil}\liminf_{n\to\infty} \norm{\pe {\scd}^nx}^{-\frac{1}{n}},\\
r^-:=&\sup_{x\in\hil}\limsup_{n\to\infty} \norm{\pe T^{n}x}^{\frac{1}{n}}.
\end{align*}
If $r^+>r^-$, then   formal Laurent  series \eqref{mod} represent analytic functions on annulus $\ann(r^-,r^+)$.
\end{theorem}
\begin{proof}
Fix $x\in\hil$. An application of the root test \cite[page 199]{rud} shows that the radius of convergence of  the regular part of the series \eqref{mod} is
\begin{equation*}
    R(x)=\liminf_{n\to\infty} \norm{\pe {\scd}^nx}^{-\frac{1}{n}},
\end{equation*}
and the radius of convergence of  the principal part of this series is
\begin{equation*}
    r(x)=\limsup_{n\to\infty} \norm{\pe T^{n}x}^{\frac{1}{n}}.
\end{equation*}
 This implies that the regular part and principal part are convergent for every $x\in \hil$  in the disc $\disc(r^+)$ and in the set  $\comp\setminus\disc(r^-)$,
respectively.
This completes the proof.
\end{proof}
As will be shown below, if the series \eqref{mod} is convergent in $E$ on $\Omega\subset\comp$ for every $x\in \hil$, then $\chil$ is a  reproducing  kernel  Hilbert  space  of  vector-valued  holomorphic  functions on $\Omega$.

\begin{theorem}\label{jad}
Suppose \eqref{li} holds and the series \eqref{mod} is convergent in $E$ on an annulus $\ann(r^-,r^+)$ with $r^-<r^+$ and  $r^-,r^+\in[0,\infty)$ for every $x\in\hil$. Then  $\chil$ is a  reproducing kernel  Hilbert  space  of $E$-valued  holomorphic  functions on $\ann(r^-,r^+)$. The reproducing kernel $\jad:\ann(r^-,r^+)\times\ann(r^-,r^+)\to \boue$ associated with $\chil$ is given by
\begin{align}\label{kern}
     \jad(z,\lambda)&=\sum_{i,j\Ge1}\pe T^{i}T^{*j}|_E\frac{1}{z^i}\frac{1}{{\bar{\lambda}}^j}+\sum_{i\Ge1,j\Ge0}\pe T^{i}{\cd}^{j}|_E\frac{1}{z^i}{\bar{\lambda}}^j\\&+\sum_{i\Ge0,j\Ge1}\pe {\scd}^{i}T^{*j}|_E{z^i}\frac{1}{{\bar{\lambda}}^j}
     +\sum_{i,j\Ge0}\pe {\cd}^{*i}{\cd}^{j}|_Ez^i{\bar{\lambda}}^j.\notag
 \end{align}
for any $z,\lambda \in \ann(r^-,r^+)$. Moreover, the following assertions hold.

\begin{itemize}
    \item[(i)]  For any $\lambda \in \annr$
    \begin{align}\label{sz1}
 &\sum_{n=1}^\infty
(P_ET^{n})\frac{1}{\lambda^n}+   \sum_{n=0}^\infty
(\pe{T^{*\prime n}})\lambda^n \in \boldsymbol B(\hil,E), \\\label{sz2}
&\sum_{n=1}^\infty
 T^{*n}\frac{1}{{\lambda}^n}+
\sum_{n=0}^\infty
 {T^{\prime n}}\lambda^n\in \boldsymbol B(E,\hil),
\end{align}
\item[(ii)] The series \eqref{kern}, \eqref{sz1} and \eqref{sz2} converges absolutely and uniformly in operator norm on any compact set contained in $\annr\times\annr$, $\annr$ and $\annr$, respectively.
\item[(iii)]
the function $\ann(r^-,r^+)\ni\lambda\to \jad(\cdot,\bar{\lambda})e\in\chil$, $e\in E$ is holomorphic and given by
\begin{align*}
    \jad(\cdot,\bar{\lambda})e=\sum_{n=1}^\infty
 UT^{*n}e\frac{1}{{\lambda}^n}+
\sum_{n=0}^\infty
 U{T^{\prime n}}e\lambda^n, \qquad \lambda \in \ann(r^-,r^+).
\end{align*}
\end{itemize}
\end{theorem}
\begin{proof}
We claim that series 
\begin{equation*}
    \sum_{n=0}^\infty
(\pe{T^{*\prime n}})\lambda^n,
\end{equation*}
converges absolutely and uniformly in the norm of $\boldsymbol B(\hil,E)$ on any compact set contained in $\ann(r^-,r^+)$.
Fix $r<r^+$.  It follows from our assumptions on the series in
\eqref{mod}
 that series
\begin{equation*}
    \sum_{n=0}^\infty
(\pe{T^{*\prime n}}x)r^n,
\end{equation*}
converges for every $x\in \hil$. Thus, there exists a constant
$C(r,x)>0$ such that
\begin{equation*}
    \|(\pe{T^{*\prime n}}x)r^n\|<C(r,x), \qquad n\in \natu.
\end{equation*}
By uniform boundedness principle (see \cite[Theorem 2.6]{rud}) we obtain that there exists a constant $M(r)>0$ such that
\begin{equation*}
    \|(\pe{T^{*\prime n}})r^n\|<M(r), \qquad n\in \natu.
\end{equation*}
If $|\lambda|<r$, then applying the above, we see that
\begin{align*}
     \|\sum_{n=0}^\infty
(\pe{T^{*\prime n}})\lambda^n\|&\Le\sum_{n=0}^\infty
\|(\pe{T^{*\prime n}})\lambda^n\|\Le \sum_{n=0}^\infty
\|(\pe{T^{*\prime n}})r^n\|\Big[\frac{|\lambda|}{r}\Big]^n\\
&\Le M(r)\sum_{n=0}^\infty
\Big[\frac{|\lambda|}{r}\Big]^n.
\end{align*}
This proves our claim. Following steps analogous to those above, we obtain that 
\begin{equation*}
    \sum_{n=1}^\infty
(P_ET^{n})\frac{1}{\lambda^n}
\end{equation*}
also converges absolutely and uniformly in the norm of $\boldsymbol B(\hil,E)$ on any compact set contained in $\ann(r^-,r^+)$. It follows from what has  already been proved that the same conclusion holds also for the series in \eqref{sz1}. This implies that the series \eqref{sz2} converges absolutely and uniformly in the norm of $\boldsymbol B(E,\hil)$ on any compact set contained in $\ann(r^-,r^+)$. Since the operator \eqref{kern} is a composition of the operators in \eqref{sz1} and \eqref{sz2}, the assertions (i) and (ii) are justified.

Let $\lambda \in \ann(r^-,r^+)$ and $e\in E$. Then
  \begin{align*}
       \la f(\lambda),e\ra_E&=\la \sum_{n=1}^\infty
(P_ET^{n}U^{-1}f)\frac{1}{\lambda^n}+
\sum_{n=0}^\infty
(P_E{T^{\prime*n}}U^{-1}f)\lambda^n,e\ra_E\\
&=\la U^{-1}f, \sum_{n=1}^\infty
T^{*n}e\frac{1}{\bar{\lambda}^n}+
\sum_{n=0}^\infty
{T^{\prime n}}e\bar{\lambda}^n\ra_E\notag
   \end{align*}
for any $f\in\chil$. As a consequence, we obtain
\begin{equation}\label{row1}
    \jad(\cdot,\lambda)=U\big(\sum_{n=1}^\infty
T^{*n}\frac{1}{\bar{\lambda}^n}+
\sum_{n=0}^\infty
{T^{\prime n}}\bar{\lambda}^n\big).
\end{equation}
This implies that

\begin{align*}
    \la \jad&(z,\lambda)e_0,e_1\ra_E=\la \jad(\cdot,\lambda)e_0,\jad(\cdot,z)e_1\ra_\chil\\
    &=\Big\la \sum_{n=1}^\infty
T^{*n}e_0\frac{1}{\lambda^n}+
\sum_{n=0}^\infty
{T^{\prime n}}\lambda^ne_0, \sum_{n=1}^\infty
T^{*n}e_1\frac{1}{z^n}+
\sum_{n=0}^\infty
{T^{\prime n}}e_1z^n\Big\ra_E,
\end{align*}
 for $e_0,e_1\in E$ and thus $\chil$ is a  reproducing kernel  Hilbert  space  of $E$-valued  holomorphic  functions on $\ann(r^-,r^+)$ and the reproducing kernel is given by \eqref{kern}.
 
 The assertion (iii) is a direct consequence of \eqref{row1} and (ii).
This completes
the proof.
\end{proof}


Now, we turn to the properties of the Cauchy dual operator $\cd$. The Cauchy dual operator $\cd$ of a left-invertible operator $T$ is itself left-invertible. Assume now that there exist a subspace $E\subset\hil$ such that $\gwon=\hil$ and $[E]_{\cd,T}=\hil$ hold. Then for both operators $T$ and $\cd$ we can construct Hilbert spaces $\chil$ and $\chil^\prime$  of $E$-valued Laurent series. Then
\begin{equation*}
    U^\prime_x(z) :=\sum_{n=1}^\infty
(P_ET^{\prime n}x)\frac{1}{z^n}+
\sum_{n=0}^\infty
(P_E(T^{*n}x)z^n.
\end{equation*}
and  $\chil'$ is the space of Laurent series of the form $U_x'$, $x\in\hil$.
\begin{theorem}
Let $T\in\bou$ and $E\subset\hil$ be a closed subspace such that $\gwon=\hil$, $[E]_{\scd,T}=\hil$ and  \eqref{prep} holds. 
Let $f$ and $g$ be $E$-valued series
\begin{equation*}
    f(z)=\sum_{n=0}^{\infty}a_nz^n\quad\textrm{and}\quad g(z)=\sum_{n=0}^{\infty}b_nz^n.
\end{equation*}
Then 
\begin{equation*}
    \la U^{-1}f,{U^{\prime -1}}g\ra=\sum_{n=0}^{\infty}\la a_n,b_n\ra.
\end{equation*}

\end{theorem}
\begin{proof}
It suffices to consider the case $f(z)=e_0z^n$ and $g(z)=e_1z^m$, $m,n\in \natu$, $e_0,e_1\in E$. Observe that
\begin{equation*} \la U^{-1}f,U^{\prime-1}g\ra= \left\{ \begin{array}{ll}
\la T^{n-m}e_0,e_1\ra & \textrm{if $n\Ge m$,}\\
\la {\cd}^{m-n}e_0,e_1\ra & \textrm{otherwise.} 
\end{array} \right.
\end{equation*}
Since \eqref{prep}, we deduce that $\la U^{-1}f,U^{\prime-1}g\ra=\delta_{m-n}\la e_0,e_1\ra$. This finishes  the proof.

\end{proof}

Now we use analytic model constructed in this section to discuss
spectral theory of left-invertible operators and its adjoints.
 \begin{theorem}\label{spec} Suppose \eqref{li} holds and the series \eqref{mod} is convergent  in $E$ for every $x\in\hil$ on open nonempty subset $\Omega\subset\comp$. Then the following assertions hold:
 \begin{itemize}
     \item[(i)] the point spectrum of $T$ is empty, that is $\sigma_p(T)=\emptyset$,
     \item[(ii)]
     $\mul^*\jad(\cdot,\mu) g=\bar{\mu}\jad(\cdot,\mu ) g$, for every $\mu\in\Omega$, $g \in E$,
     \item[(iii)] $\bar{\Omega}\subset\sigma_p(T^*)$,
     \item[(iv)] $\bigvee\set{\nul T^*-\bar{\mu})\colon \mu \in U}=\hil$, where $U\subset \Omega$ and $\intt U\neq \emptyset$.
 \end{itemize}
 \end{theorem}
 \begin{proof}
 
(i) Suppose, to derive a contradiction, that $\mu\in\sigma_p(T)$. Then by Theorem  \ref{przep}, $\mu\in\sigma_p(\mul)$. Let $f(z)=\sum_{n=-\infty}^{\infty}a_nz^n\in \chil$ be such that 
 $(\mul-\mu)f=0$. Using the properties of reproducing kernel, one gets the following
 \begin{equation*}
     \la (\mul-\mu)f,\jad(\cdot,\lambda)e \ra=0, \qquad \lambda\in\Omega,\: e\in E,
 \end{equation*}
 Since the series $f(z)=\sum_{n=-\infty}^{\infty}a_nz^n$ is convergent on $\Omega$, we
see that the above equality is equivalent to the following one
\begin{equation*}
 (\lambda-\mu)\sum_{n=-\infty}^{\infty}\la a_n,e\ra \lambda^n=0,\qquad \lambda\in\Omega,\: e\in E.
 \end{equation*}
 If $\mu=0$, then by Identity theorem
 \begin{equation*}
      \la a_n,e\ra = 0,\qquad e\in E,\: n \in \mathbb{Z}.
 \end{equation*}
 This means that $a_n=0$ for $ n \in \mathbb{Z}$. This shows
that $f = 0$, which gives (i).
We now consider the other case when $\mu\neq0$. Using Identity theorem again, we see that
 \begin{equation}\label{geo}
 \la a_n,e\ra = \mu^{-n}\la a_0,e\ra,\qquad n \in \mathbb{Z}.
 \end{equation}
 Suppouse that there exist some $e\in E$ such that $\la a_0,e\ra\neq0$. 
 Therefore, by \eqref{geo}
 \begin{equation*}
     \sum_{n=-\infty}^{\infty}\la a_n,e\ra \lambda^n=\la a_0,e\ra\sum_{n=-\infty}^{\infty}\Big(\frac{\lambda}{\mu}\Big)^n,\qquad \lambda\in\Omega,\: e\in E.
 \end{equation*}
Clearly, $\sum_{n=-\infty}^{\infty}(\frac{z}{\lambda})^n$ is divergent, which contradicts
our assumption that $\la a_0,e\ra\neq0$. This and \eqref{geo} shows that $a_n=0$ for $ n \in \mathbb{Z}$. As a consequence, we get $f=0$. This completes the proof of (i).

(ii) By Theorem \ref{jad}, we have
 \begin{align*}
     \la U_x,\mul^*\jad(\cdot,\lambda) g\ra&=\la \mul U_x,\jad(\cdot,\lambda )g\ra=\la \lambda U_x(\lambda), g\ra\\&=\la  U_x(\lambda),\bar{\lambda} g\ra=\la U_x,\bar{\lambda}\jad(\cdot,\lambda ) g\ra,
 \end{align*}
 for $x\in\hil$,  $\lambda \in \Omega$ and $g\in E$. This gives the equality
 \begin{equation*}
     \mul^*\jad(\cdot,\lambda) g=\bar{\lambda}\jad(\cdot,\lambda ) g.
 \end{equation*}
 
 (iii) This is a direct consequence of (ii).
 
 (iv) Suppouse that $f(z)=\sum_{n=-\infty}^{\infty}a_nz^n\in \chil$  is orthogonal to the subspace $\bigvee\set{\nul \mul^*-\bar{\mu})\colon \mu \in U}$. Since, by (ii), $\jad(\cdot,\mu) e\in \nul \mul^*-\bar{\mu})$ for every $e\in E$, the following equalities hold 
 \begin{equation*}
     \sum_{n=-\infty}^{\infty}\la a_n,e\ra \lambda^n=\la f(\lambda),e\ra=\la f,\jad(\cdot,\lambda)e\ra=0,\qquad \lambda\in U,\: e\in E.
 \end{equation*}
 By Identity theorem this implies that $a_n=0$ for every $n\in \mathbb{Z}$. Thus $f=0$ and $\bigvee\set{\nul \mul^*-\bar{\mu})\colon \mu \in U}=\chil$.
 \end{proof}
 \section{Weighted composition operators.} 
 In this section as an application of the model presented in Section \ref{mymodel}, we obtain significantly improved model for weighted composition operator upon provided the symbol of this operator
has finite branching index.

 We begin by recalling the definition of finite branching index.
 Let $\ttt=(V,E)$  be a rootless directed tree. Following \cite{chav}, we say that $\ttt$ has \textit{finite branching index} if there exist $m\in \natu$ such that
 \begin{equation*}
     \czil^k(V_\prec)\cap V_\prec=\emptyset,\quad k\Ge m,\:k\in\natu.
 \end{equation*}
 The next lemma shows that in the case of  rootless
directed tree with finite branching index
  there exist some special vertex.
\begin{lemma}[(\cite{chav})]\label{groot}
 Let $\ttt=(V,E)$  be a rootless directed tree with finite branching
index $m$. Then there exist a vertex $\omega\in V_\prec$ such that
\begin{equation}\label{root}
\card(\czil(\parr^{(n)}(\omega)))=1,\qquad n\in\mathbb{Z}_+.
\end{equation}
Moreover, if $V_\prec$ is non-empty, then there exists a unique $\omega\in V_\prec$ satisfying \eqref{root}.
 \end{lemma} The vertex $\omega\in V_\prec$ appearing in the statement of Lemma \ref{groot} is called
 \textit{generalized root}. We put $x^*:=\parr(\omega)$ in the definition of function $[\varphi]:X\rightarrow \mathbb{Z}$ for orbit $F$ of $\varphi$
not containing a cycle (see Section \ref{prel}).
 Since any self-map $\varphi:X\to X$ induces  a directed graph $(X, E^\varphi)$ (see Figure \ref{fig:M5}) given by 
 \begin{equation}\label{graph}
     E^\varphi=\set{(x,y)\in X\times X \colon x=\varphi(y)}
 \end{equation}
 it is natural to extend the notion of finite branching index to self-maps.
 We say that $\varphi$ has \textit{finite branching index} if \begin{equation*}\sup \set{|[\varphi](x)|: \card(\varphi^{-1}(x))\Ge2,\:\: x\in X}< \infty.
 \end{equation*}
 Perhaps it is appropriate at this point to note that a self-map with one orbit can have at most one cycle.

  \begin{figure}
     \centering
 \begin{tikzpicture}

 \begin{scope}[every node/.style={fill=gray!20,circle,thick, minimum width=30pt}]
    \node (A) at (2,2) {$x_{m-1}$};
    \node (B) at (3.41,1.41) {$x_m$};
    \node (C) at (4,0) {$x_0$};
    \node (D) at (3.41,-1.41) {$x_1$};
    \node (E) at (2,-2) {$x_2$};
    \node (F) at (0.59, -1.41) {$x_3$};
    \node (G) at (5.41,1.41) {$x_{0,1}$};
    \node (H) at (7,1.41) {$x_{0,2}$};
    \node[fill=white] (I) at (9.5,1.41) {};
    \node (J) at (5.41,-1.41) {$x_{1,1}$};
    \node (K) at (7,-1.41) {$x_{1,2}$};
    \node[fill=white] (L) at (9.5,-1.41) {};
\end{scope}

\begin{scope}[>={Stealth[black]},
              every node/.style={fill=white,circle},
              every edge/.style={draw=black,very thick}]
    \path [->] (A) edge[bend left=20] (B);
    \path [->] (B) edge[bend left=20]  (C);
    \path [->] (C) edge[bend left=20]  (D);
    \path [->] (D) edge[bend left=20]  (E);
    \path [->] (E) edge[bend left=20]  (F);
    \path [->] (F) edge[bend left=60,loosely dotted]  (A);
    \path [->] (G) edge  (C);
    \path [->] (H) edge  (G);
    \path [->] (I) edge[loosely dotted]  (H);
    \path [->] (J) edge  (C);
    \path [->] (K) edge  (J);
    \path [->] (L) edge[loosely dotted]  (K);
\end{scope}
\end{tikzpicture}
 \caption{} \label{fig:M5}
\end{figure}

 Recall that  a weighted shift on a rootless directed tree can be identified with
composition operator in $L^2$-spaces (see \cite[Lemma 4.3.1]{9}).

 Let $X$ be a countable set, $w:X\to \comp$ be a complex function on $X$,  $\varphi:X\to X$ be a transformation of $X$ and $\com$ be a weighted composition operator in $\ell^2(X)$.
We will need only consider
composition functions with one orbit, since an orbit induces a reducing subspace
to which the restriction of the weighted composition operator is again a weighted
composition operator.

The following lemma describes a  subspace $E\subset \ell^2(X)$ of the operator $\com$ which satisfies condition \eqref{prep} with $\com$ and $\cdcom$ in place of $T$.  It requires
considering two distinct cases.
\begin{lemma}\label{wond} Let $X$ be a countable set, $w:X\to \comp$ be a complex
function on $X$ and  $\varphi:X\to X$ be a transformation of $X$,  which  has  finite  branching  index. Let $\com$ be a weighted composition operator in $\ell^2(X)$ and
\begin{equation}\label{eee}
E: = \left\{ \begin{array}{ll}
\bigoplus_{x\in\gen_\varphi(1,1)}\la e_x\ra\oplus\nul (\com|_{\ell^2(\des(x))})^*) & \textrm{when $\varphi$ has  a  cycle,}
\\
\la e_\omega\ra\oplus\nul\com^*) & \textrm{otherwise,}
\end{array} \right.
\end{equation}
where  $\des(x):=\bigcup_{n=0}^\infty\varphi^{(-n)}(x)$ and $\omega$ is a generalized root of the tree defined by \eqref{graph}. Then the subspace
$E$ has
the following properties:
\begin{itemize}
\item[(i)]$\cgwon=\hil$ and $[E]_{{\com},{\cdcom^*}}=\hil$,
\item[(ii)] $E\perp \com^n E$ and $E\perp \cdcom^n E$,  $n\in \mathbb{Z}_+$.
\end{itemize}
\end{lemma}
\begin{proof}
(i) First, we consider the case when $\varphi$ does not have a cycle.  Clearly, the weighted composition operator $\com$ can be identified with
 a weighted shift $\szift$
on a rootless directed tree given by \eqref{graph}.
We show that the subspace $E:=\la e_\omega\ra\oplus\nul{\szift}^*)$ satisfies \eqref{prep} and $\sgwon=\ell^2(X)$. Note  that the space $\ell^2(\des(\omega))$
 is invariant for $\szift$. We will denote by $\szifto$ the operator $\szift|_{\ell^2(\des(\omega))}$. The subtree $\ttt_{\des(\omega)}$ of $\ttt$ is a directed tree with root $\omega$ and by
Lemma \ref{ker}, $\nul\szifto)= \la e_\omega\ra\oplus\nul{\szift}^*)$. Since by \cite[Lemma 3.3]{chav} $\szifto$ is analytic, it follows from Shimorin's analytic model that $[E]_{\szifto}=\ell^2(\des(\omega))$. Hence,
\begin{equation}\label{fp2}\bigvee\{\szift^n x\colon x\in E, n\in \natu\}=[E]_{\szifto}=\ell^2(\des(\omega)).
\end{equation}

By \cite[Proposition 2.7]{shi}, we have $[E]_{(\szifto)^\prime}=\ell^2(\des(\omega))$. Note that  the subspace $\ell^2(\des(\omega))$ is invariant for $\szift$ and $\szift^*\szift$ is diagonal. 
Recall that if $T\in\bou$ and closed subspace $\mathcal{G}$ is  invariant for $T$, then $(T|_\mathcal{G})^*=P_\mathcal{G}T^*|_\mathcal{G}$. Therefore, we have
\begin{align*}
(\szifto)^\prime&=\szift|_{\ell^2(\des(\omega))}((\szift|_{\ell^2(\des(\omega))})^*\szift|_{\ell^2(\des(\omega))})^{-1}\\
&=\szift|_{\ell^2(\des(\omega))}(P_{\ell^2(\des(\omega))}\szift^*|_{\ell^2(\des(\omega))}\szift|_{\ell^2(\des(\omega))})^{-1}\\
&=\szift(\szift^*\szift)^{-1}|_{\ell^2(\des(\omega))}=\cdszift|_{\ell^2(\des(\omega))}.
\end{align*}
This implies that 
\begin{equation}\label{fpzz}\bigvee\{\cdszift^n x\colon x\in E, n\in \natu\}=[E]_{(\szift|_{\ell^2(\des(\omega))})^\prime}=\ell^2(\des(\omega)).
\end{equation}
The assertion (i) of Lemma \ref{podst}, shows that
\begin{align*}\bigvee\{\szift^{*n} x\colon x\in E, n\in \natu\}&=\ell^2(X\setminus\des(\omega)),
\\\bigvee\{\cdszift^{*n} x\colon x\in E, n\in \natu\}&=\ell^2(X\setminus\des(\omega)).
\end{align*}
This
together with \eqref{fp2} and \eqref{fpzz}  yields
$\sgwon=\ell^2(X)$ and 
$[E]_{\szift,\cdszift^*}=\ell^2(X)$.

If $\varphi$  has  a  cycle, then the operator
\begin{equation*}\com|_{\ell^2(\des(x))} \quad \textrm{for}\quad x\in \gen_\varphi(1,1)
\end{equation*}
is a weighted shift on directed tree with root $x$.  Using
\cite[Lemma 3.4]{chav} again and  arguing as in the previous case we obtain 
\begin{align}\label{cz1}
\bigvee\{\cdcom^n y\colon y\in \la e_x\ra \oplus\nul (\com|_{\ell^2(\des(x))})^*), n\in  \natu\}&=\ell^2(\des(x)),\:\: x\in\gen_\varphi(1,1),\\
\bigvee\{\com^n y\colon y\in \la e_x\ra \oplus\nul (\com|_{\ell^2(\des(x))})^*), n\in  \natu\}&=\ell^2(\des(x)),\:\: x\in\gen_\varphi(1,1).\notag
\end{align}
Applying the assertion (i) of Lemma \ref{podst}
 again, we see that
\begin{align*}
\bigvee\{\com^n x\colon x\in E, n\in \natu\}&=\ell^2(X\setminus\bigcup_{x\in\gen_\varphi(1,1)}\des(x)),\\\bigvee\{\cdcom^n x\colon x\in E, n\in \natu\}&=\ell^2(X\setminus\bigcup_{x\in\gen_\varphi(1,1)}\des(x)).
\end{align*}
This and \eqref{cz1}, implies that $[E]_{\com^*,{\cdcom}}=\hil$ and $[E]_{\com,\cdcom^*}=\hil$.


(ii) First, we consider the case when $\varphi$
does not have a cycle. According to Lemma \ref{ker}, $E=\la e_\omega \ra\oplus \nul \com^*)$. If $e,f\in E$ and $n\in\natu$, then
\begin{align*}\la \com^{*n} e,  f\ra&=\la \com^{*n}P_{\la e_\omega \ra}e ,f \ra=0,\\
\la \cdcom^{*n} e,  f\ra&=\la \cdcom^{*n}P_{\la e_\omega \ra}e ,f \ra=0,
\end{align*}
where in the last step we used   assertion (iii) of Lemma \ref{podst}.
 This immediately yields that condition \eqref{prep} holds, which completes the proof
of the case when $\varphi$ does  not  have  a  cycle.

 If $\varphi$  has  a  cycle,  then similar reasoning leads to the equalities
 \begin{align*}\la \com^{*n} e,  f\ra&=\la \com^{*n}P_{\tilde{E}}e ,f \ra=0,\\
\la \cdcom^{*n} e,  f\ra&=\la \cdcom^{*n}P_{\tilde{E}}e ,f \ra=0,
\end{align*}
where $\tilde{E}:=\bigvee\{e_x\colon x\in \gen_\varphi(1,1)\}$.
 This completes the proof. 
\end{proof}


Before we turn to the
main theorem of this section, we need to give some definitions.
Suppose \eqref{li} holds with $\com$ in place of $T$.  
Let $\varphi$ be a self-map of $X$ and $E$ be a subspace of $\ell^2(X)$. Define 
\begin{equation*}
k_{\varphi}(E):=\min\{n\in \natu\colon E\subset\bigvee\{e_x\colon\gen_\varphi(1,n) \}\}.
\end{equation*}
A number $k_{\varphi}(E)$ will be called an \textit{index} of $E$ with respect to $\varphi$. Now, we can define some  subsets of $X$ by
\begin{equation*}
W^{E,\varphi}_0:=\gen_\varphi(1,k_{\varphi}(E)),
\end{equation*}
and then
\begin{equation*} W^{E,\varphi}_n:=\left\{\begin{array}{ll}
\varphi^{(-n)} (W^{E,\varphi}_0) & n\in \natu\textrm{ when $\varphi$ has a cycle}\\
\varphi^{(-n)}(W^{E,\varphi}_0) &n\in \mathbb{Z}\textrm{ otherwise.}
\end{array} \right.
\end{equation*}
Finally, we are ready to define radii of convergence for $\com$. The non-negative number
\begin{equation}\label{outer}
 \rplus :=\liminf_{n \to \infty}\Big( \hspace{-0.3cm} \sum_{\substack{x\in  W^{E,\varphi}_n\\n\Ge0}}\hspace{-0.3cm}|w^\prime(x)w^\prime(\varphi(x))\cdots w^\prime(\varphi^{(n-1)}(x))|^2\Big)^{-\frac{1}{2n}}
\end{equation}
will be called the \textit{outer radius of convergence} for $\com$, and similarly the non-negative number 
\begin{equation}\label{inner}
\rmin:= \left\{ \begin{array}{ll}
\sqrt[\tau]{\prod_{x\in\cycle}|w(x)|} & \textrm{if $\varphi$ has a cycle,}\\ \limsup_{n\to \infty}
\sqrt[n]{|w(\varphi^{1}(\omega))w(\varphi^{2}(\omega))\dots w(\varphi^{n}(\omega))|} & \textrm{otherwise},
\end{array} \right.
\end{equation}
where $\tau:=\card \cycle$ will be called the \textit{inner radius of convergence} for $\com$.

Now we are in a position to prove the main result of this section (compare with \cite[Theorem 2.2]{chav}).
\begin{theorem}\label{modcom} Let $X$ be a countable set, $w:X\to \comp$ be a complex
function on $X$ and  $\varphi:X\to X$ be a transformation of $X$,  which  has  finite  branching  index. Let $\com$ be a left-invertible weighted composition operator in $\ell^2(X)$. If $\rplus>\rmin$, then there exist a $z$-invariant reproducing kernel Hilbert
space $\mathscr{H}$ of $E$-valued holomorphic functions defined on the annulus $\anns$ and a unitary
mapping $U:\ell^2(X)\rightarrow\mathscr{H} $ such that $\mathscr{M}_zU=U\com$, where $\mathscr{M}_z$ denotes the operator
of multiplication by $z$ on $\mathscr{H}$, where $E$ is   as in \eqref{eee}. Moreover,  the following assertions hold
:
\begin{itemize}
    \item[(i)] the reproducing kernel $\jad:\anns\times\anns\to \boue$ associated with $\chil$ has the property that $\jad(\cdot,w)g\in \chil$ and $\la Uf,\jad(\cdot,w)g\ra=\la(Uf)(w),g\ra$ for $f,g\in\ell^2(X)$.
    \item[(ii)] the reproducing kernel $\jad$ has the following form:
    \begin{align*}
     \jad(z,\lambda)=\sum_{i,j\Ge1}A_{i,j}\frac{1}{z^i}\frac{1}{\lambda^j}&+\sum_{i\Ge1,j\Ge0}B_{i,j}\frac{1}{z^i}\lambda^j
     \\
     &+\sum_{i\Ge0,j\Ge1}C_{i,j}z^i\frac{1}{\lambda^j}
     +\sum_{i,j\Ge0}D_{i,j}z^i\lambda^j,
 \end{align*} where $A_{i,j}, B_{i,j}, C_{i,j}, D_{i,j}\in\sbou(E)$; if additionally $\varphi$ has no cycle, then
 \begin{align*}
A_{i,j}&=0 & &\textrm{if} & |i-j|&>\kk,\\
B_{i,j}&=0 & &\textrm{if} & i+j&>\kk,\\
C_{i,j}&=0 & &\textrm{if} & i+j&>\kk,\\
D_{i,j}&=0 & &\textrm{if} & |i-j|&>\kk.
\end{align*}

 \item[(iii)] if $\varphi$
does
not have a cycle, then
 the linear subspace generated by
$E$-valued polynomials in
$z$
and $\tilde{E}$-valued polynomials involving only 
negative powers of $z$ is dense in $\chil$, that is
\begin{equation*}
\bigvee(\{z^nE\colon n\in \natu\}\cup\{\frac{1}{z^n}\tilde{E} \colon n\in \mathbb{Z}_+\})=\chil,
\end{equation*}
where $\tilde{E}:=\bigvee\{e_x\colon x\in \gen_\varphi(1,1)\}$;
if $\varphi$
 has a cycle $\cycle$ with $\tau:=\card\cycle$, then there exist $\tau$ functions $f_1,\dots, f_\tau$ on $\anns$ given by the following Laurent series
 \begin{equation*}
   f_i(z):=\sum_{k=0}^\infty\sum_{i=1}^{\tau} \Lambda^kA_i\frac{1}{z^{k\tau+i}}, \quad i=1,...,\tau,
 \end{equation*}
where $\Lambda:=\prod_{x\in\cycle}w(x)$ such that the linear subspace generated by $E$-valued polynomials in $z$ and the above functions is dense in $\chil$, that is
    \begin{equation*}
\bigvee(\{z^nE\colon n\in \natu\}\cup\{f_i \colon i\in\set{1,\dots \tau}\})=\chil.
\end{equation*}

\end{itemize}
\end{theorem}
\begin{proof} We begin by showing that the $E$-valued series
\begin{equation*}  \sum_{n=0}^\infty\pe\cdcom^{*n}fz^n
\end{equation*} converges absolutely in $E$ on the disc $\disc(\rplus)$.
Let $f=\sum_{x\in X} f(x)e_x$. Applying Lemma \ref{podst}, we obtain
\begin{align*}
    \pe\cdcom^{*n}f&=\sum_{x\in X}f(x)w^\prime(x)w^\prime(\varphi(x))\cdots w^\prime(\varphi^{(n-1)}(x))\pe e_{\varphi^{(n)}(x)}\\&=\sum_{x\in  W^{E,\varphi}_n }f(x)w^\prime(x)w^\prime(\varphi(x))\cdots w^\prime(\varphi^{(n-1)}(x))\pe e_{\varphi^{(n)}(x)}.
\end{align*}
Observe that $W^{E,\varphi}_n \cap W^{E,\varphi}_m=\emptyset$ for $|m-n|>\kk$, $m,n\in\natu$. As a consequence, we have
\begin{equation}\label{rozl}
    \sum_{\substack{x\in  W^{E,\varphi}_n\\ n\Ge 0} }|f(x)|^2\Le(\kk+1)\sum_{x\in X}|f(x)|^2=(\kk+1)\|f\|^2.
\end{equation}
By the Cauchy-Schwarz inequality, we have
\begin{align*}
    \|\sum_{n=0}^k\pe\cdcom^{*n}fz^n\|\hspace{-0.1cm}&\Le\hspace{-0.3cm}\sum_{\substack{x\in  W^{E,\varphi}_n\\ n\Ge 0} }|f(x)w^\prime(x)w^\prime(\varphi(x))\cdots w^\prime(\varphi^{(n-1)}(x))z^n|\\&\Le \Big(\hspace{-0.3cm}\sum_{\substack{x\in  W^{E,\varphi}_n\\ n\Ge 0}}\hspace{-0.3cm}|f(x)|^2\Big)^\frac{1}{2}\hspace{-0.1cm}\Big(\hspace{-0.3cm}\sum_{\substack{x\in  W^{E,\varphi}_n\notag\\ n\Ge 0}}\hspace{-0.3cm}|w^\prime(x)w^\prime(\varphi(x))\cdots w^\prime(\varphi^{(n-1)}(x))z^n|^2\Big)^\frac{1}{2}\\&\overset{\eqref{rozl}}{\hspace{-0.1 cm}\Le}\hspace{-0.3 cm}\sqrt{k_{\varphi}(E)+1}\norm{f}\Big(\hspace{-0.3cm}\sum_{\substack{x\in  W^{E,\varphi}_n\notag\\n\Ge
    0}}\hspace{-0.4cm}|w^\prime(x)w^\prime(\varphi(x))\cdots w^\prime(\varphi^{(n-1)}(x))z^n|^2\Big)^\frac{1}{2}.\notag
\end{align*}
An application of the root test \cite[page 199]{rud} shows that the above series converges on the disc $\disc(\rplus)$.

Now we show that the $E$-valued series
\begin{equation*}
   \sum_{n=0}^\infty\pe\com^{n}f \frac{1}{z^n}
\end{equation*}
converges absolutely in $E$ on  $\comp\setminus\disc(\rmin)$.
First,  we  consider  the  case  when $\varphi$ does  not  have  a  cycle. For this,
 note that using Lemma \ref{podst} again,
\begin{align*}\pe\com^{n}f&=\sum_{x\in X}f(x)
\sum_{y\in\varphi^{-n}(x)}w(y)w(\varphi(y))\cdots w(\varphi^{(n-1)}(y))\pe e_y
\\&=\hspace{-0.2cm}\sum_{x\in  W^{E,\varphi}_{-n} }f(x)\sum_{y\in\varphi^{-n}(x)}w(y)w(\varphi(y))\cdots w(\varphi^{(n-1)}(y))\pe e_y
\\&=\hspace{-0.2cm}\sum_{x\in  W^{E,\varphi}_{-n} }f(x)w(\varphi^{-1}(x))w(\varphi^{-2}(x))\cdots w(\varphi^{[\varphi](x)}(x))\pe \com^{n+[\varphi](x)}e_{\varphi^{[\varphi](x)}(x)},
\end{align*}
for $n\Ge\kk$. Put $M:= \max\{1,\norm{\com}^{\kk}\}$.
Note that $n+[\varphi](x)\Le \kk$  and thus, $\|\pe \com^{n+[\varphi](x)}\|\Le M$.  Repeating the argument in \eqref{rozl}, we see that
\begin{equation*}
    \sum_{\substack{x\in  W^{E,\varphi}_{-n}\\ n\Ge \kk}}|f(x)|^2\Le(\kk+1)\|f\|^2.
\end{equation*} Hence, by the Cauchy-Schwarz inequality again, we have
\begin{align*}
    \|\sum_{n=\kk}^k&\pe\com^{n}f\frac{1}{z^n}\|\hspace{-0.1cm}\Le M\sum_{\substack{x\in  W^{E,\varphi}_{-n}\\ n\Ge\kk}}|f(x)w(\varphi^{-1}(x))w(\varphi^{-2}(x))\cdots w(\varphi^{[\varphi](x)}(x))\frac{1}{z^n} |\\
    &\Le M \Big(\hspace{-0.3cm}\sum_{\substack{x\in  W^{E,\varphi}_{-n}\\ n\Ge \kk}}\hspace{-0.3cm}|f(x)|^2\Big)^\frac{1}{2}\hspace{-0.1cm}\Big(\hspace{-0.3cm}\sum_{\substack{x\in  W^{E,\varphi}_{-n}\notag\\ n\Ge  \kk}}\hspace{-0.3cm}|w(\varphi^{-1}(x))w(\varphi^{-2}(x))\cdots w(\varphi^{[\varphi](x)}(x))\frac{1}{z^n}|^2\Big)^\frac{1}{2}
    \\&\Le M\sqrt{k_{\varphi}(E)+1}\norm{f}\Big(\hspace{-0.3cm}\sum_{\substack{x\in  W^{E,\varphi}_{-n}\notag\\n\Ge
    \kk}}\hspace{-0.3cm}w(\varphi^{-1}(x))w(\varphi^{-2}(x))\cdots w(\varphi^{[\varphi](x)}(x)\frac{1}{z^n}|^2\Big)^\frac{1}{2},\notag
\end{align*}
for $k\Ge\kk$.
Since the series on the right-hand side converges absolutely on $\comp\setminus\disc(\rmin)$,
we are done.
It remains to consider the other case when $\varphi$ has a cycle. It is easily seen that
\begin{equation*}
   \sum_{n=0}^\infty\pe\com^{n}f z^n = \sum_{n=0}^\infty\pe\com^{n}P_{\hil_{\varphi}}f z^n+\sum_{n=0}^\infty\pe\com^{n}P_{\hil_{\varphi}^\perp} f z^n,
\end{equation*}
where $\hil_{\varphi}:=\lin\set{e_x\colon x\in \cycle}$.
We show that both above series converge. Observe that
if $h\in \hil_{\varphi}^\perp \cap\bigvee \set{e_x\colon x\in \gen_\varphi(m,n)}$ for $m,n\in \natu$, then $\com h\in \bigvee\set{e_x\colon x\in \gen_\varphi{(m+1,n+1)}}
$. This, together with  the fact that $E\subset \bigvee\{e_x \colon x \in \gen_\varphi{(1,\kk)}\}$ yields
\begin{align*}
  \Big \| \sum_{n=0}^\infty\pe\com^{n}P_{\hil_{\varphi}^\perp}f z^n \Big\|= \Big\| \sum_{n=0}^{k_{\varphi}(E)}\pe\com^{n}P_{\hil_{\varphi}^\perp}f z^n \Big\|\Le \sum_{n=0}^{k_{\varphi}(E)}\norm{\com}^{n}\norm{f} z^n.
\end{align*}
Let us now observe that
\begin{equation}\label{okr}
P_{W_i}\com^{n+\tau} e_x=\Lambda P_{W_i}\com^{n}e_x, \quad  x\in \cycle,
\end{equation}
where $ W_i=\bigvee\set{e_x\colon x\in\gen_\varphi(i,i) }
$, $i\in\natu$. We now apply this observation to
estimate the following sum.
\begin{align*}
 \sum_{n=0}^\infty\pe&\com^{n}P_{\hil_{\varphi}}f\frac{1}{z^n}=\sum_{n=0}^\infty\pe\com^{n}\Big(\sum_{x\in\cycle}f(x)e_x\Big)\frac{1}{z^n}
 =\sum_{x\in\cycle}f(x)\sum_{n=0}^\infty\pe\com^{n}e_x\frac{1}{z^n}\\
 &=\sum_{x\in\cycle}\sum_{i=0}^{\kk}f(x)\sum_{n=0}^\infty\pe P_{W_i}\com^{n}e_x\frac{1}{z^n}\\
 &=\sum_{x\in\cycle}\sum_{i=0}^{\kk}\sum_{j=0}^\tau f(x)\sum_{n=0}^\infty\pe P_{W_i}\com^{n\tau+j}e_x\frac{1}{z^{n\tau+j}}\\
 &=\sum_{x\in\cycle}\sum_{i=0}^{\kk}\sum_{j=0}^\tau f(x)\sum_{n=0}^\infty \Lambda^n\pe P_{W_i}\com^{j}e_x\frac{1}{z^{n\tau+j}}.
 \end{align*}
The above series is a finite sum of a geometric
series with ratio $\frac{\Lambda}{z^\tau}$ and hence converges  absolutely on $\comp\setminus\disc(\rmin)$.
Putting
these results together,  we conclude that if $\rplus>\rmin$, then the series \ref{mod} with $\com$ in place of $T$ converges absolutely on 
$\anns$. Moreover, combining Theorems \ref{przep} and \ref{jad} with Lemma \ref{wond}, we deduce that there exist a $z$-invariant reproducing kernel Hilbert
space $\mathscr{H}$ of $E$-valued holomorphic functions defined on the annulus $\anns$ and a unitary
mapping $U:\ell^2(X)\rightarrow\mathscr{H} $ such that $\mathscr{M}_zU=U\com$.

Now we turn to the proof of the "moreover" part.

(i) This assertion is a direct consequence of Theorem \ref{jad}.

(ii) Recall that by \eqref{kern}, the kernel has the following form
\begin{align*}
     \jad(z,\lambda)&=\sum_{i,j\Ge1} A_{i,j}\frac{1}{z^i}\frac{1}{\lambda^j}+\sum_{i\Ge1,j\Ge0}B_{i,j}\frac{1}{z^i}\lambda^j\\&+\sum_{i\Ge0,j\Ge1} C_{i,j}{z^i}\frac{1}{\lambda^j}
     +\sum_{i,j\Ge0}D_{i,j}z^i\lambda^j,\notag
 \end{align*}
 where 
\begin{align*}
    A_{i,j}&=\pe \com^{i}\com^{*j}|_E,  & B_{i,j}&=\pe \com^{i}{\cdcom}^{j}|_E,\\
    C_{i,j}&=\pe \cdcom^{*i}\com^{*j}|_E, & D_{i,j}&=\pe \cdcom^{*i}\cdcom^{j}|_E.
\end{align*}
Observe that 
\begin{align*} 
\cdcom^{*m}\cdcom^n E&\subset\bigvee\{e_x\colon x\in W^{E,\varphi}_{n-m}\},&
\com^m\cdcom^n E&\subset\bigvee\{e_x\colon x\in  W^{E,\varphi}_{m+n}\},\\
\cdcom^{*m}\com^{*n} E&\subset\bigvee\{e_x\colon x\in  W^{E,\varphi}_{-m-n}\},&
\com^m\com^{*n} E&\subset\bigvee \{e_x\colon x\in W^{E,\varphi}_{m-n}\}.
\end{align*}
Since $E\subset \bigvee\{e_x\colon x\in  W^{E,\varphi}_{0}\}$, the subspace $E$ is orthogonal to $\bigvee\{e_x\colon x\in  W^{E,\varphi}_{k}\}$ if $|k|>\kk$. This completes the proof of (ii).

(iii) It follows from  Lemma \ref{wond} that
\begin{equation*}
\bigvee(\{\com^nE\colon n\in \natu\}\cup\{\com^{*n} E\colon n\in \natu\})=\hil.
\end{equation*}
We now consider two disjunctive cases which cover all possibilities.
First we consider the case when $\varphi$
does not have a
cycle.
Since $\com$ is unitarily equivalent to $\mul$, we see
that 
\begin{equation}\label{fp}
    U(\bigvee\{\com^nE\colon n\in \natu\})=\bigvee \{\mul^n(E)\colon n\in \natu\}.
\end{equation}
 Note that 
 \begin{equation}\label{spr}
     U(\cdcom^{*n}e_\omega)=\Big(\prod_{i=0}^{n-1}w(\varphi^{(i)}(\omega)) \overline{w^\prime(\varphi^{(i)}(\omega)}\Big)e_\omega\frac{1}{z^n},\qquad n\in\mathbb{Z}_+.
 \end{equation}It follows from \eqref{eee} and equality $\nul\cdcom^*)=\nul \com^*)$ that
 \begin{equation*}
     \bigvee\{\cdcom^{*n} E\colon n\in \natu\}=\bigvee\{\cdcom^{*n} e_\omega\colon n\in \natu\}.
 \end{equation*}
Combining this with \eqref{fp} and \eqref{spr} completes the proof
of the case when $\varphi$
does not have a cycle.
 It remains to consider the other case when $\varphi$ has a cycle.
Looking at the formula \eqref{eee},  we  deduce  that
\begin{equation}\label{znik}
\cdcom^{*} e = \left\{ \begin{array}{ll}
 \overline{ w^\prime(x)}e_{\varphi(x)} & \textrm{if $e=e_x$,  $x\in\gen_\varphi(1,1)$}\\
0 & \textrm{if $e\in\bigoplus_{x\in\gen_\varphi(1,1)}\nul (\com|_{\ell^2(\des(x))})^*)$}
\end{array} \right.
\end{equation}
Note that if $x\in\gen_\varphi(1,1)\cup\cycle$, then $\varphi(x)\in \cycle$.
 This, combined with \eqref{znik}, yields
 \begin{equation*} \bigvee\{\cdcom^{*n} E\colon n\in \mathbb{Z}_+\}=\bigvee\{e_x\colon x \in \cycle\}.
 \end{equation*}
 We now describe the value of the map $U:\hil\to\chil$ at $e_x$, $x\in\cycle$. 
In view of \eqref{okr},  we  can  deduce  from \eqref{mod} that $U(e_x)$, $x\in \cycle$ has the following form
\begin{equation*}
U(e_x)=\sum_{k=0}^\infty\sum_{i=1}^{\tau} \Lambda^kA^{x}_i\frac{1}{z^{k\tau+i}},
\end{equation*}
for some $A^{x}_i\in \tilde{E}$, $i=1,\dots,\tau$.  This completes the proof.

\end{proof}

\section{Examples}
In this section, we illustrate Theorem  \ref{modcom} 
by considering  several interesting
examples. 
We begin by giving an example of left-invertible  weighted composition operator  
for which the series in \eqref{mod} does  not  converge absolutely on any open subset of $\comp$.
\begin{ex}Fix $m\in \natu$ and set $X=\set{0,1,\ldots,m}$. Let $w:X\to\comp$ be a function and define a mapping $\varphi:X\to X$ by
\begin{equation*}
\varphi(i) = \left\{ \begin{array}{ll}
i+1 & \textrm{if $i<m$}\\
0 & \textrm{if $i=m$}
\end{array} \right.
\end{equation*}
(see Figure \ref{fig:M1}). Set $\Lambda:=w(0)w(1)\dots w(m)$.  Let $\com$ be the left-invertible composition
operator in $\comp^{m+1}$.
\begin{figure}
\centering
\begin{tikzpicture}
\begin{scope}[every node/.style={fill=gray!20,circle,thick}]
    \node (A) at (5.5,1.5) {$x_{m}$};
    \node (B) at (6.8,0.75) {$x_0$};
    \node (C) at (6.8,-0.75) {$x_1$};
    \node (D) at (5.5,-1.5) {$x_2$};
    \node (E) at (4.3, -0.75) {$x_3$};
\end{scope}

\begin{scope}[>={Stealth[black]},
              every node/.style={fill=white,circle},
              every edge/.style={draw=black,very thick}]
    \path [->] (A) edge[bend left=20] (B);
    \path [->] (B) edge[bend left=20]  (C);
    \path [->] (C) edge[bend left=20]  (D);
    \path [->] (D) edge[bend left=20]  (E);
    \path [->] (E) edge[bend left=50,loosely dotted]  (A);
\end{scope}
\end{tikzpicture}
\caption{} 
\label{fig:M1}
\end{figure}
The matrix of this operator is of the form
\begin{equation*}
    \com =
\left[ \begin{array}{cccccc}
0 & w(0) & 0 & \cdots  & 0\\
0 & 0 & w(1) & \cdots  & 0\\
0 & 0 & 0 & \cdots  & 0\\
\vdots & \vdots &\vdots & \ddots  & \vdots\\
w(m) & 0 & 0& \cdots & 0
\end{array} \right].
\end{equation*}
Let $E:=\lin\set{e_1}$. It is easy to see that $[E]_{\bil,\bil^{\prime*}}=\hil$. Using Lemma \ref{podst} and Lemma \ref{cdcom}, one can o verify that
\begin{align*}
\pe\com^{mk+r}x&=\Lambda^{k}\Big(\prod_{i=0}^{r-1}w(i)\Big) x_{r}e_0, \\
\pe\com^{\prime*(mk+r)}x&=\frac{1}{\Lambda^k}\Big(\hspace{-0.3 cm}\prod_{i=m+1-r}^{m}\hspace{-0.3 cm}w(i)\Big)^{-1}x_{n+1-r}e_0,
\end{align*}
for $r<n$, $r,k\in \natu$.
This shows that  formal Laurent series in \eqref{mod} takes the following form:
\begin{align*} U_x(z) &=\sum_{k=1}^\infty\sum_{r=0}^{n-1}
\Big(\Lambda^{k}\Big(\prod_{i=0}^{r-1}w(i)\Big) x_{r}e_0\Big)\frac{1}{z^{nk+r}}
\\&+
\sum_{k=0}^\infty
\Big(\sum_{r=0}^{n-1}\frac{1}{\Lambda^k}\Big(\hspace{-0.3 cm}\prod_{i=m+1-r}^{m}\hspace{-0.3 cm}w(i)\Big)^{-1}x_{n+1-r}e_0\Big )z^{nk+r}.
\end{align*}
Since $\com^*$ acts on the finite dimensional space, the spectrum of $\com^*$ is finite. Therefore, by assertion (iii) of Theorem \ref{spec} the above series does not  converge absolutely on any open
subset of $\comp$. Alternatively, one can prove this fact directly by calculating convergences radii.

\end{ex}

The next example shows that our analytic model generalises the Gellar's analytic model for bilateral weighted shift  \cite{gel}.  
\begin{ex}\label{pokgel}(Bilateral weighted shift)

Let $\bil:\ell^2(\mathbb{Z})\to\ell^2(\mathbb{Z})$ be a bilateral weighted shift with weights $\{\lambda_n \}_{n \in
\mathbb{Z}}$ and $\{e_n \}_{n \in
\mathbb{Z}}$ be the standard orthonormal basis of $\ell^2(\mathbb{Z})$. Then
\begin{equation*}\bil e_n=\lambda_{n+1} e_{n+1},\qquad n\in\mathbb{Z}
\end{equation*}
(see Figure \ref{fig:M2}). Let $E:=\lin\set{e_0}$. It is easy to see that $[E]_{\bil^*,\bil^{\prime}}=\hil$. It is worth noting that $\nul \bil^*)=\set{0}$ and thus $[\nul \bil^*)]_{\bil^*,\bil^{\prime}}=\set{0}$. This phenomenon is quite different
comparing with the case of left-invertible and analytic operators in which $[\nul T^*)]_{T^*,\cd}=\hil$, where $T$ is in this class.

\begin{figure}
\centering
\begin{tikzpicture}
\begin{scope}[every node/.style={fill=gray!20,circle,thick}]
    \node[fill=white] (A) at (1.5,0) {};
    \node (B) at (3,0) {$x_{-2}$};
    \node (C) at (4.5,0) {$x_{-1}$};
    \node[minimum width=27pt] (D) at (6,0) {$x_0$};
    \node[minimum width=27pt] (E) at (7.5, 0) {$x_1$};
    \node[minimum width=27pt] (F) at (9, 0) {$x_2$};
    \node[fill=white] (G) at (10.5, 0) {};
\end{scope}

\begin{scope}[>={Stealth[black]},
              every node/.style={fill=white,circle},
              every edge/.style={draw=black,very thick}]
    \path [->](B)  edge[loosely dotted](A) ;
    \path [->] (C) edge[] (B)  ;
    \path [->] (D) edge[] (C) ;
    \path [->] (E) edge[] (D) ;
     \path [->] (F) edge[] (E) ;
      \path [->](G)  edge[loosely dotted] (F) ;
\end{scope}
\end{tikzpicture}
\caption{}
\label{fig:M2}
\end{figure}

It is a matter of routine to verify that the Cauchy dual $\bil^{\prime*}$ of $\bil$ has the following form
\begin{equation*}
\bil^{\prime*}e_n=\frac{1}{\lambda_n}e_{n-1}, \qquad n\in \mathbb{Z}.
\end{equation*}
It
is now easily seen that 
\begin{equation*}\pe(\bil^{\prime*})^nx=\Big(\prod_{i=1}^n\lambda_i\Big)^{-1}x_ne_0,\qquad n\in \mathbb{Z}_+,
\end{equation*}
and
\begin{equation*}\pe\bil^nx=\Big(\hspace{-0.1cm}\prod_{i={-n+1}}^0\lambda_i\Big)x_{-n}e_0,\qquad n\in \mathbb{Z}_+.
\end{equation*}
Therefore, by \eqref{mod} the formal Laurent series takes the
form
\begin{align*} 
U_x(z) =\sum_{n=1}^\infty\Big(\prod_{i={-n+1}}^0\lambda_i\Big)x_{-n}\frac{1}{z^n}+
\sum_{n=0}^\infty\Big(\prod_{i=1}^n\lambda_i\Big)^{-1}x_nz^n.
\end{align*}
Comparing the above series with the formal Laurent series   in \cite[Section 2]{gel}  one can realize that our analytic model and the Gellar  analytic model coincide in the case of left-invertible   bilateral weighted shifts. Noting that
$W^{E,\varphi}_n=\{n\}$ for $n\in \mathbb{Z}$, we infer from \eqref{outer} and \eqref{inner} that 
\begin{equation*}
 \rplus =\liminf_{n \to \infty}\sqrt[n]{\prod_{i=1}^n|\lambda_i|}
\end{equation*}
and
\begin{equation*}
\rmin=  \limsup_{n\to \infty}
\sqrt[n]{\prod_{i={-n+1}}^0|\lambda_i}|.
\end{equation*}
In this case, the reproducing kernel $\jad:\annr\times\annr\to\boue$ is diagonal and given by
\begin{equation*}
     \jad(z,\lambda)=\sum_{i=1}^\infty \prod_{i={-n+1}}^0|\lambda_i|^2\frac{1}{(z\bar{\lambda})^i}
     +\sum_{i=0}^\infty\Big(\prod_{i=1}^n|\lambda_i|^2\Big)^{-1}(z\bar{\lambda})^i.
 \end{equation*}
\end{ex}
Now we provide two more examples of left-invertible compositions operators over connected  directed  graphs  induced  by  self-maps  whose  vertices,  all  but  one,
have valency one and the valency of the remaining vertex is nonzero.
\begin{ex}Set $m\in \natu$ and $X=\set{0,1,\dots m}\sqcup \set{(0,i)\colon i\in \natu}$. Let $w:X\to \comp$ be a measurable function and  $\varphi:X\to X$ be transformation of $X$ defined by
\begin{equation*}
\varphi(x )= \left\{ \begin{array}{ll}
(0,i-1) & \textrm{for $x=(0,i)$,  $i\in\natu\setminus\set{0}$},\\
m & \textrm{for $x=(0,0)$},\\
i-1 & \textrm{for $x=i$ and $i\in \{1,\dots,m\}$,}\\m&\textrm{for $x=0$,}
\end{array} \right.
\end{equation*}
(see Figure \ref{fig:M3}). Let $\com:\ell^2(X)\to\ell^2(X)$ be a left-invertible composition operator. It is easily seen that
\begin{equation*}\com e_x= \left\{ \begin{array}{ll}
w({(0,i+1)})e_{(0,i+1)} & \textrm{for $x=(0,i)$,  $i\in\natu\setminus\set{0}$}\\
w(i+1)e_{i+1} & \textrm{for $x=i$ and $i\in\{0,1,\dots,m\}$}\\
w(0)e_0+w({(0,0)})e_{(0,0)}&\textrm{for $x=m$.}
\end{array} \right.
\end{equation*}
\begin{figure}
\centering
\begin{tikzpicture}
 \begin{scope}[every node/.style={fill=gray!20,circle,thick, minimum width=30pt}]
    \node (A) at (2,2) {$x_{1}$};
    \node (B) at (3.41,1.41) {$x_0$};
    \node (C) at (4,0) {$x_m$};
    \node (D) at (3.41,-1.41) {$x_{m-1}$};
    \node (E) at (2,-2) {$x_{m-2}$};
    \node (F) at (0.59, -1.41) {$x_{m-3}$};
    \node (G) at (5.5,0) {$x_{0,0}$};
    \node (H) at (7,0) {$x_{0,1}$};
    \node[fill=white] (I) at (9.5,0) {};
\end{scope}
\begin{scope}[>={Stealth[black]},
              every node/.style={fill=white,circle},
              every edge/.style={draw=black,very thick}]
    \path [->] (A) edge[bend left=20] (B);
    \path [->] (B) edge[bend left=20]  (C);
    \path [->] (C) edge[bend left=20]  (D);
    \path [->] (D) edge[bend left=20]  (E);
    \path [->] (E) edge[bend left=20]  (F);
    \path [->] (F) edge[bend left=60,loosely dotted]  (A);
    \path [->] (G) edge  (C);
    \path [->] (H) edge  (G);
    \path [->] (I) edge[loosely dotted]  (H);
\end{scope}
\end{tikzpicture}
\caption{} \label{fig:M3}
\end{figure}It is routine to verify that  $\nul \com^*)=\lin\{ \overline{w({(0,0)})}e_{0}-\overline{w({0})}e_{(0,0)}\}$. Let $E:=\lin\{ e_{(0,0)}\}$. One can check that this one-dimensional subspace satisfies \eqref{li}.
This implies that\footnote{To make the notation more readable, we adopt the convention that $\prod_{i=n}^{m}a_i=1$ if $n>m$.}
\begin{align*}\pe(\cdcom^*)^nx&=\Big(\prod_{i=1}^nw{(0,i)}\Big)^{-1}x_ne_{(0,0)}, \\
\pe\com^{nm+r+1}x&=\Lambda^{n}w({(0,0)})\Big(\prod_{i=0}^{r-1}w({m-i})\Big)x_{m-r}e_{(0,0)},
\end{align*}
for $r<m$, $r,n\in \natu$. Hence, by \eqref{mod} the Hilbert space $\chil$ consist of the functions of the form
\begin{align*}
U_x(z) &=\sum_{n=1}^\infty\sum_{r=0}^{k}
\Lambda^{k}w({(0,0)})\Big(\prod_{i=0}^{r-1}w({m-i})\Big)x_{m-r}\frac{1}{z^{nm+r+1}}\\&+
\sum_{n=0}^\infty
\Big(\prod_{i=1}^nw({(0,i)})\Big)^{-1}x_nz^{n}.\notag
\end{align*}
The formulas for the inner and outer radius of convergence take the following form
 \begin{equation*}
 \rplus =\liminf_{n \to \infty}\sqrt[n]{\prod_{i=1}^n|w({(0,i)})|}
\end{equation*}
and
\begin{equation*}
\rmin=\sqrt[m+1]{ \prod_{i={0}}^m|w(i)|}.
\end{equation*}
The reproducing kernel $\jad:\annr\times\annr\to\boue$ by Theorem \ref{jad} takes the form
 \begin{align*}
     \jad(z,\lambda)&=\sum_{i\Ge1,j\Ge1} \Lambda^{i}\bar{\Lambda}^j|w({(1,0)})|^2\Big(\prod_{i=0}^{r-1}|w({m-i})|^2\Big)\frac{1}{z^{im+r+1}\bar{\lambda}^{jm+r+1}}
     \\&+\sum_{i=0}^\infty\Big(\prod_{i=1}^n|w({(0,i)})|^2\Big)^{-1}(z\bar{\lambda})^i.
 \end{align*}
\end{ex}
In
Example \ref{ost}
below,  we  demonstrate composition
operator which  can be identified with weighted shift on a rootless directed tree. A weighted shift on a
directed tree is a circular operator \cite[Theorem 3.3.1.]{memo}. Hence, without loss of generality we can
assume that the weight is positive.
\begin{ex}\label{ost}
Set $m\in \natu$ and $X=\natu\sqcup \set{(i,j)\colon i\in\set{0,1}, j\in \natu }$. Let $w:X\to (0,\infty)$ be a measurable function and  $\varphi:X\to X$ be transformation of $X$ defined by
\begin{equation*}
\varphi(x )= \left\{ \begin{array}{ll}
(i,j-1) & \textrm{for $x=(i,j)$,  $i\in\mathbb{Z}_+$, $j\in\set{0,1}$},\\
0 & \textrm{for $x\in\{(0,0),(1,0)\}$},\\
x+1 & \textrm{for $x\in\natu$},
\end{array} \right.
\end{equation*}
 \begin{figure}
     \centering
 \begin{tikzpicture}
 \begin{scope}[every node/.style={fill=gray!20,circle,thick, minimum width=30pt}]
    \node (A) at (0,0) {$x_{2}$};
    \node (B) at (2,0) {$x_1$};
    \node (C) at (4,0) {$x_0$};
    \node[fill=white] (F) at (-2,0) {};
    \node (G) at (5.41,1.41) {$x_{0,0}$};
    \node (H) at (7,1.41) {$x_{0,1}$};
    \node[fill=white] (I) at (9.5,1.41) {};
    \node (J) at (5.41,-1.41) {$x_{1,0}$};
    \node (K) at (7,-1.41) {$x_{1,1}$};
    \node[fill=white] (L) at (9.5,-1.41) {};
\end{scope}
\begin{scope}[>={Stealth[black]},
              every node/.style={fill=white,circle},
              every edge/.style={draw=black,very thick}]
    \path [->] (B) edge[] (A);
    \path [->] (C) edge[]  (B);
    \path [->] (A) edge[loosely dotted]  (F);
    \path [->] (G) edge  (C);
    \path [->] (H) edge  (G);
    \path [->] (I) edge[loosely dotted]  (H);
    \path [->] (J) edge  (C);
    \path [->] (K) edge  (J);
    \path [->] (L) edge[loosely dotted]  (K);
\end{scope}
\end{tikzpicture}
 \caption{} \label{fig:M4}
\end{figure}
(see Figure \ref{fig:M4}). Let $\com\in \boldsymbol B(\ell^2(X))$ be a left-invertible composition
operator. As an immediate consequence of the definition, we obtain
\begin{equation*}\com e_x= \left\{ \begin{array}{ll}
w({(i,j+1)})e_{(i,j+1)} & \textrm{for $x=(i,j)$,  $i\in\set{0,1}$, $j\in\natu$,}\\
w({i-1})e_{i-1} & \textrm{for $x=i$ and $i\in \mathbb{Z}_+$,}\\
w(1,0)e_{(1,0)}+w({(1,0)})e_{(1,0)}&\textrm{for $x=0$.}
\end{array} \right.
\end{equation*}
 Applying Lemma \ref{ker}, we get
$\nul\com^*)=\lin\{w({0,0})e_{(1,0)}-w({1,0})e_{(0,0)}\}$. Therefore, by Lemma \ref{wond}
$E=\lin\{e_0,w({2,0})e_{(1,0)}-w({2,0})e_{(1,0)}\}$.
By more or less elementary
calculations,  one can verify that
\begin{align*}
\pe\com^{n}x&=\Big(\prod_{i=0}^{n-1}w(i)\Big)x_ne_0,\\
\pe(\cdcom^*)^nx&=W\Big[w(0,0)\Big(\prod_{i=1}^{n-1}w{(0,i)}\Big)^{\hspace{-0.1cm}-1}\hspace{-0.2cm}x_{(0,n-1)}+w{(1,0)}\Big(\prod_{i=1}^{n-1}w{(1,i)}\Big)^{\hspace{-0.1cm}-1}\hspace{-0.2cm}x_{(1,n-1)}\Big]e_0
\\&+\Big[\Big(\prod_{i=1}^{n}w{(1,i)}\Big)^{-1}w{(1,0)}x_{(0,n)}-\Big(\prod_{i=1}^{n}w{(1,i)}\Big)^{-1}w{(0,0)}x_{(1,n)}\Big]\tilde{e},
\end{align*}
for $n\in \mathbb{Z}_+$, where 
\begin{equation*}
 W=\frac{1}{w^2{(0,0)}+w^2{(1,0)}},\qquad   \tilde{e}=\frac{w{(1,0)}e_{(0,0)}-w{(0,0)}e_{(1,0)}}{w^2{(0,0)}+w^2{(1,0)}}.
\end{equation*}
Therefore, by \eqref{mod} the formal Laurent series takes the
form
\begin{align*} 
U_x(z) &=\sum_{n=1}^\infty\Big(\prod_{i=0}^{n-1}w(i)\Big)x_ne_0\frac{1}{z^n}\\&+
\sum_{n=0}^\infty W\Big[w(0,0)\Big(\prod_{i=1}^{n-1}w{(0,i)}\Big)^{-1}\hspace{-0.2cm}x_{(0,n-1)}+w{(1,0)}\Big(\prod_{i=1}^{n-1}w{(1,i)}\Big)^{-1}\hspace{-0.2cm}x_{(1,n-1)}\Big]e_0\\&+\Big[\Big(\prod_{i=1}^{n}w{(1,i)}\Big)^{-1}w{(1,0)}x_{(0,n)}-\Big(\prod_{i=1}^{n}w{(0,i)}\Big)^{-1}w{(0,0)}x_{(1,n)}\Big]\tilde{e}]z^n
\end{align*}
We infer from \eqref{outer} and \eqref{inner} that the formulas for the inner and outer radius of convergence take the following form
\begin{equation*}
 \rplus :=\liminf_{n \to \infty}\Big(\sum_{i=0}^1(\frac{1}{w^2(i,0)}+\frac{1}{w^2(i,n+1)})\prod_{j=1}^n\frac{1}{w^2(i,j)}\Big)^{-\frac{1}{2n}}
\end{equation*}
and
\begin{equation*}
\rmin:=  \limsup_{n\to \infty}
\sqrt[n]{\prod_{i={0}}^{n-1}|w(i)|}.
\end{equation*}
\end{ex}














\textbf{Acknowledgements.}
I am very grateful to Professor Jan Stochel for his valuable comments, as well as substantial help he provided
me while working on this paper.
\bibliographystyle{amsalpha}
   
\end{document}